\documentclass[a4paper,11pt]{article}
\usepackage{amsmath,amsthm}
\usepackage{aligned-overset}
\usepackage{authblk}
\usepackage[style=numeric-comp,maxbibnames=99,bibencoding=utf8,giveninits=true]{biblatex}
\usepackage{booktabs}
\usepackage{cancel}
\usepackage{cases}
\usepackage{comment}
\usepackage[hmargin={26mm,26mm},vmargin={30mm,35mm}]{geometry}
\usepackage{hyperref}
\usepackage[compact]{titlesec}
\usepackage{graphicx,epstopdf}
\usepackage{subcaption}
\graphicspath{{figures/}}
\usepackage{xcolor}

\usepackage{newtxtext}
\usepackage{newtxmath}

\bibliography{p-adr}

\newtheorem{theorem}{Theorem}

\newtheorem{lemma}[theorem]{Lemma}

\theoremstyle{remark}
\newtheorem{remark}[theorem]{Remark}
\theoremstyle{definition}

\newcommand{\email}[1]{\href{mailto:#1}{#1}}

\newcommand{\ud}{\mathrm{d}}

\DeclareMathOperator{\card}{card}

\newcommand{\Real}{\mathbb{R}}
\newcommand{\st}{\;|\;}

\newcommand{\Poly}[1]{\mathcal{P}^{#1}}

\newcommand{\elements}[1]{\mathcal{T}_{#1}}
\newcommand{\faces}[1]{\mathcal{F}_{#1}}

\newcommand{\Th}{\elements{h}}
\newcommand{\Thd}{\elements{h}^{\rm d}}
\newcommand{\Tha}{\elements{h}^{\rm a}}
\newcommand{\Fh}{\faces{h}}
\newcommand{\Fhb}{\faces{h}^{\rm b}}
\newcommand{\Fhi}{\faces{h}^{\rm i}}
\newcommand{\Fhd}{\faces{h}^{\rm d}}
\newcommand{\Fha}{\faces{h}^{\rm a}}
\newcommand{\TT}{\elements{T}}
\newcommand{\TF}{\elements{F}}
\newcommand{\FT}{\faces{T}}

\newcommand{\rF}{r_F^{k}}
\newcommand{\Rh}{R_h^{k}}
\newcommand{\Gh}{G_h^{k}}

\newcommand{\normal}{n}

\newcommand{\jump}[1]{[#1]_F}
\newcommand{\avg}[1]{\{#1\}_F}

\newcommand{\Erra}{\mathcal{E}_{a,h}^k}
\newcommand{\Errb}{\mathcal{E}_{b,h}^k}
\newcommand{\Err}{\mathcal{E}_h^k}

\newcommand{\lproj}[2]{\pi_{#1}^{#2}}

\newcommand{\term}{\mathfrak{T}}

\newcommand{\norm}[2]{\|#2\|_{#1}}
\newcommand{\Norm}[2]{\left\|#2\right\|_{#1}}
\newcommand{\seminorm}[2]{|#2|_{#1}}

\makeatletter
\newcommand{\tnorm}{\@ifstar\@tnorms\@tnorm}
\newcommand{\@tnorms}[2]{%
  \left|\mkern-1.5mu\left|\mkern-1.5mu\left|
  #2
  \right|\mkern-1.5mu\right|\mkern-1.5mu\right|_{#1}
}
\newcommand{\@tnorm}[2]{%
  \mathopen{|\mkern-1.5mu|\mkern-1.5mu|}
  #2
  \mathclose{|\mkern-1.5mu|\mkern-1.5mu|}_{#1}
}
\makeatother

\newcommand{\Pe}{\mathrm{Pe}}

\newcommand{\vref}[1]{\hat{\beta}_{#1}}

\newcommand{\dref}[1]{\hat{K}_{#1}}

\newcommand{\tref}{\hat{\tau}_T}

\newcommand{\err}{{\rm err}}

\begin{document}

\title{A P\'eclet-robust discontinuous Galerkin method for nonlinear diffusion with advection}
\author[1,2]{Louren\c{c}o Beir\~{a}o da Veiga}
\affil[1]{Dipartimento di Matematica e Applicazioni, Universit\`{a} di Milano Bicocca, 
\email{lourenco.beirao@unimib.it}, \ \email{k.haile@campus.unimib.it}}
\affil[2]{IMATI-PV, CNR, Pavia, Italy}
\author[3]{Daniele A. Di Pietro}
\affil[3]{IMAG, Univ. Montpellier, CNRS, Montpellier, France, \email{daniele.di-pietro@umontpellier.fr}}
\author[1]{Kirubell B. Haile}

\maketitle

\begin{abstract}
  We analyze a Discontinuous Galerkin method for a problem with linear advection-reaction and $p$-type diffusion, with Sobolev indices $p\in (1, \infty)$. 
  The discretization of the diffusion term is based on the full gradient including jump liftings and interior-penalty stabilization while, for the advective contribution, we consider a strengthened version of the classical upwind scheme. The developed error estimates track the dependence of the local contributions to the error on local P\'eclet numbers. A set of numerical tests supports the theoretical derivations.
\end{abstract}

\section{Introduction}

Discontinuous Galerkin (DG) methods were introduced in the 70s \cite{Reed.Hill:73,Baker:77} and have gained significant popularity starting from the late 90s \cite{Cockburn.Shu:91,Castillo.Cockburn.ea:00,Arnold.Brezzi.ea:01,Castillo.Cockburn.ea:02,Di-Pietro.Ern:10,Burman.Ern:08,Di-Pietro.Ern:12,Bassi.Botti.ea:12,Bassi.Botti.ea:14,Antonietti.Giani.ea:13}.
They are nowadays widely regarded as the reference methods for advection-dominated problems.
When a polynomial degree $k\ge 1$ is used, classical error estimates for linear diffusion-advection(-reaction) problems show that the error contribution stemming from diffusive terms is $\mathcal{O}(h^k)$ (with $h$ denoting the meshsize), while the one stemming from advective terms is $\mathcal{O}(h^{k+\frac12})$; see, e.g., \cite{Ayuso.Marini:09} and also \cite{Di-Pietro.Ern.ea:08} and \cite[Section 4.6]{Di-Pietro.Ern:12} for an analysis covering the locally degenerate case.
Pre-asymptotic convergence rates between $k$ and $k+\frac12$ can be observed, in practice, when sufficiently coarse meshes are considered.
Standard estimates do not usually allow, however, a quantitative assessment of this phenomenon.
Error estimates are, on the other hand, completely missing for problems with non-linear diffusion terms.

The goal of this work is to fill the above gaps by deriving P\'eclet-dependent error estimates for a problem with linear advection-reaction and $p$-type diffusion, for Sobolev indices $p\in (1, \infty)$.
The discretization of the diffusion term is, similarly to \cite{Burman.Ern:08,Del-Pezzo.Lombardi.ea:12}, based on the full gradient including jump liftings and interior-penalty stabilization.
For the advective contribution, on the other hand, we consider a strengthened version of the classical upwind scheme obtained interpreting the latter as a penalty contribution in the spirit of \cite{Brezzi.Marini.ea:04}.
The peculiarity of our error estimates is that they track the dependence of the local contributions to the error on local P\'eclet numbers. To improve the estimates of certain terms, we provide a new extension to the nonconforming case of the techniques of \cite{Hirn:13}, based in turn on the results of \cite{Diening.Ettwein:08} (see also \cite{Barrett.Liu:93}).
This requires a certain number of subtleties, both in the adaptation of the argument and in the definition of the face P\'eclet numbers (which need to account for both the physical and numerical diffusion).
To the best of our knowledge, our P\'eclet-dependent error estimates are the first of this kind for a nonlinear problem, and enable a quantitative assessment of pre-asymptotic convergence rates.
In the linear case, corresponding to $p = 2$, local P\'eclet numbers can be computed based on the sole knowledge of the problem data and the mesh, making it possible to identify a priori advection- and diffusion-dominated elements/faces.
Incidentally, new error estimates for the DG discretization of the $p$-Laplace problem are also recovered as a special case (the previous works \cite{Burman.Ern:08,Del-Pezzo.Lombardi.ea:12} only considered convergence by compactness).
The theoretical results are supported by extensive numerical validation.

The present contribution furthermore sets the stage for future publications developing pressure robust and advection-robust finite elements for time-dependent Navier--Stokes type equations (e.g. \cite{Han.Hou:21,Beirao-da-Veiga.Dassi.ea:23}) modeling incompressible fluid flows with non-Newtonian rheology.

The rest of this work is organized as follows.
In Section \ref{sec:cont} we describe the continuous problem. After presenting some definitions and preliminary results in Section \ref{sec:sett}, the numerical scheme is introduced in Section \ref{sec:discr} along with the main theoretical results.
The proofs of the latter are given in Section \ref{sec:ana}.
Finally, numerical tests are collected in Section \ref{sec:num}.



\section{The continuous problem}\label{sec:cont}

Let $\Omega\subset\Real^d$, $d\ge 1$, denote a bounded, connected polyhedral domain.
We develop a Péclet-robust discontinuous Galerkin (DG) method for the following problem:
Find $u:\Omega\to\Real$ such that
\begin{equation}\label{eq:strong}
  \begin{alignedat}{2}
    -\nabla\cdot\left[
      \sigma(\nabla u)
      - \beta u
      \right] + \mu u &= f
    &\qquad&\text{in $\Omega$},
    \\
    u &= 0 
    &\qquad&\text{on $\partial\Omega$}.
  \end{alignedat}
\end{equation}
Here above, we assume that the velocity field satisfies $\beta\in W^{1,\infty}(\Omega)^d$ and, for the sake of simplicity, that $\nabla\cdot\beta = 0$ almost everywhere in $\Omega$.
Furthermore, we assume $\mu(x) \ge \underline{\mu} > 0$ for almost every $x \in \Omega$ with $\underline{\mu} \in \Real$.
The extension to non-incompressible velocity fields is standard, and essentially requires to assume a positive lower bound on the quantity $\mu + \frac12 (\nabla \cdot \beta)$ instead of $\mu$.
The function $\sigma$ represents the diffusive flux function, which we describe below.

For given real number $p\ge 1$ and integer $n\ge 1$, we consider the power flux function
\[
\sigma_n:\Real^n\ni x\mapsto |x|^{p-2}x\in\Real^n
\]
with $|{\cdot}|$ denoting the Euclidian norm.
In what follows, for the sake of brevity, we omit the subscript when $n = d$, i.e., we set $\sigma \coloneq \sigma_d$.
The following derivations can be extended to more general flux functions satisfying {appropriate} $p$-monotonicity and $p$-continuity properties characterizing Leray--Lions-type operators and their generalizations; see, e.g., \cite{Leray.Lions:65,Diening.Ettwein:08}.

In what follows, to alleviate the notation, we will use the symbol $c(\delta)$ for a generic constant, possibly different at each occurrence, which depends on the parameter $\delta$ but is independent of the meshsize (see below), the problem data and solution.

\begin{lemma}[Modified monotonicity of the power flux function]
  Let $p \in (1,\infty)$ and an integer $n \ge 1$ be given.
  For all $x,y,z \in \Real^n$ and any real number $\delta > 0$, it holds
  \begin{equation}\label{hirn-bound}
    (\sigma_n(x) - \sigma_n(z)) \cdot (x - y) \le
    \delta (\sigma_n(y) - \sigma_n(x))\cdot(y-x)
    + c(\delta) \big( |x| + |z| \big)^{p-2} |x-z|^2 \, ,
  \end{equation}
  with $c(\delta)$ positive constant depending only on $p$ and $\delta$.
\end{lemma}
  \begin{proof}
    Throughout this proof, $a \lesssim b$ means $a \le b$ with hidden constant only depending on $p$, while $a \simeq b$ stands for ``$a \lesssim b$ and $b \lesssim a$''.
    Let $\varphi:\Real^+ \ni t \mapsto \frac1p t^p \in \Real^+$ and, for any $a \ge 0$ and any $t \ge 0$, let 
    $\varphi_a(t) \coloneq \int_0^t \varphi'(a + s) \frac{s}{a+s} \,\ud s$.
    By \cite[Eq.~(6.28)]{Diening.Ettwein:08}, it holds, for all $s,t \in \Real^+$ and all $\delta > 0$,
    \begin{equation}\label{eq:generalized.Young}
      \varphi_a'(s)t + \varphi_a'(t)s
      \le \delta \varphi_a(s)
      + c(\delta) \varphi_a(t).
    \end{equation}
    Moreover, by \cite[Lemma~3]{Diening.Ettwein:08}, we have, for all $x,y \in \Real^n$, 
    \begin{equation}\label{eq:equivalence:1}
      (\sigma_n(y) - \sigma_n(x)) \cdot (y - x)
      \simeq \varphi_{|x|}(|x - y|)
      \simeq (|x| + |y|)^{p-2} |x - y|^2
    \end{equation}
    and, as observed in \cite[Lemma~2.3]{Hirn:13}, for all $x,z \in \Real^n$, 
    \begin{equation}\label{eq:equivalence:2}
      |\sigma_n(x) - \sigma_n(z)| \lesssim \varphi_{|x|}'(|x - z|).
    \end{equation}
    Let now $x,y,z \in \Real^n$.
    Applying \eqref{eq:generalized.Young} with $a = |x|$, $s = |x - y|$, and $t = |x - z|$, and noticing that $\varphi_a'(s)t \ge 0$, we get
    \begin{equation}\label{eq:hirn-bound:intermediate}
      \varphi_{|x|}'(|x - z|) |x - y|
      \le \delta \varphi_{|x|}(|x - y|) + c(\delta) \varphi_{|x|}(|x - z|).
    \end{equation}
    The conclusion follows first observing that $(\sigma_n(x) - \sigma_n(z)) \cdot (x - y) \le |\sigma_n(x) - \sigma_n(z)|\, |x - y|$ and using \eqref{eq:equivalence:2} to estimate the left-hand side of \eqref{hirn-bound} with the left-hand side of \eqref{eq:hirn-bound:intermediate}, then applying, respectively, the left-most equivalence in \eqref{eq:equivalence:1} to estimate $\varphi_{|x|}(|x - y|)$ and the right-most equivalence in \eqref{eq:equivalence:1} to estimate $\varphi_{|x|}(|x - z|)$ in the right-hand side of \eqref{eq:hirn-bound:intermediate}.
  \end{proof}  

\section{The discrete setting and preliminary results}\label{sec:sett}

We denote by $\Th$ a mesh of $\Omega$ belonging to an admissible sequence $\{\Th\}_h$ in the sense of \cite[Section 1.4]{Di-Pietro.Ern:12}.
For any $T\in\Th$, we denote by $\omega_T$ the union of the mesh elements sharing at least one face with $T$, which are collected in the set $\elements{T}$.
We moreover denote by $\Fh$ the set of faces, partitioned into boundary faces collected in $\Fhb$ and interfaces collected in $\Fhi$.
Given a face $F\in\Fh$, we denote by $\TF$ the set of mesh elements sharing $F$ and by $\omega_F$ their union.
Furthermore, for any $T\in\Th$ we denote by $\FT$ the set of its faces.
For any mesh element or face $Y \in \Th \cup \Fh$, we denote by $h_Y$ its diameter and set $h \coloneq \max_{T \in \Th} h_T$.
Since $\Th$ belongs to an admissible mesh sequence, the maximum number of faces of a mesh element is bounded uniformly in $h$.

To avoid naming generic constants, from this point on we will use the notation $a\lesssim b$ to express the inequality $a\le Cb$ with $C$ independent of the meshsize, of the problem data and solution, but possibly depending on other quantities including the domain, the ambient dimension $d$, the mesh regularity parameter, and the Sobolev index $p$.
We will write $a \simeq b$ in lieu of ``$a \lesssim b$ and $b \lesssim a$''.

\begin{remark}[Polytopal meshes]
  We underline that the present results apply not only to standard type of grids, but also to general polytopal meshes. For a few (among many) examples of other polytopal schemes in a similar context, see for example \cite{Beirao-da-Veiga.Dassi.ea:21,Di-Pietro.Droniou.ea:21,Beirao-da-Veiga.Dassi.ea:21*1,Di-Pietro.Droniou:23}.
\end{remark}

\subsection{Local and broken spaces}

Given a polynomial degree $k\ge 0$ and a mesh element $T\in\Th$, we denote by $\Poly{k}(T)$ the space spanned by the restriction to $T$ of $d$-variate polynomial functions.
At the global level, we define the broken polynomial space
\[
\Poly{k}(\Th)\coloneq\left\{
v_h\in L^1(\Omega)\st
\text{
  $v_T\coloneq(v_h)_{|T}\in\Poly{k}(T)$
  for all $T\in\Th$
}
\right\}.
\]
The $L^2$-orthogonal projector on $\Poly{k}(\Th)$ is denoted by $\lproj{h}{k}$ and is obtained patching together the $L^2$-orthogonal projectors $\lproj{T}{k}$ on $\Poly{k}(T)$, $T \in \Th$.
The same notations are used for vector versions of these projectors mapping on $\Poly{k}(\Th)^d$ or $\Poly{k}(T)^d$ and acting component-wise.
Letting $Y \in \Th \cup \Fh \cup \{\Omega\}$, we denote by $W^{q,p}(Y)$ the usual Sobolev space on $Y$ and we set
\[
W^{q,p}(\Th)\coloneq\left\{
v\in L^p(\Omega)\st
\text{
  $v_{|T}\in W^{q,p}(T)$ for all $T \in \Th$
}
\right\}.
\]
We will also need broken spaces defined on local patches $\elements{Y}$, $Y \in \Th \cup \Fh$, defined in a similar way.
Finally, for the Hilbertian case $p=2$, we will also use the habitual abridged notations $H^q\coloneq W^{q,2}$ and $L^2\coloneq H^0$.

For future use, for any $p\in(1,+\infty)$ we define the conjugate index $p'$ such that
  \begin{equation}\label{eq:p'}
    \frac{1}{p} + \frac{1}{p'} = 1
    \iff p' = \frac{p}{p-1}.
  \end{equation}
  The above definition can be generalized to $p\in\{1,\infty\}$ setting $\frac{1}{\infty}\coloneq 0$ and $\frac{1}{0}\coloneq \infty$.

\subsection{Trace operators and integration by parts formula}

For each interface $F\in\Fhi$, we fix once and for all an orientation for the unit normal vector $\normal_F$.
Denoting by $T_1$ and $T_2$ the elements sharing $F$ ordered so that $\normal_F$ points out of $T_1$, we define the jump and average operators such that, for any $\varphi\in W^{1,1}(\Th)$,
\[
\jump{\varphi}\coloneq\varphi_{|T_1} - \varphi_{|T_2},\qquad
\avg{\varphi}\coloneq\frac12\left(\varphi_{|T_1} + \varphi_{|T_2}\right).
\]
When applied to vector-valued functions, these operators act component-wise.
The above operators are extended to boundary faces $F\in\Fhb$ setting
\[
\jump{\varphi}
= \avg{\varphi}
\coloneq\varphi.
\]
We recall the following integration by parts formula:
For all $\tau:\Omega\to\Real^d$ and $v:\Omega\to\Real$ smooth enough,
\begin{equation}\label{eq:ibp}
  \int_\Omega\tau\cdot\nabla_h v
  = -\int_\Omega(\nabla_h\cdot\tau)~v
  + \sum_{F\in\Fhi}\int_F\avg{\tau}\cdot\normal_F~\jump{v}
  + \sum_{F\in\Fhi}\int_F\jump{\tau}\cdot\normal_F~\avg{v}
  + \int_{\partial\Omega}(\tau\cdot\normal)\,v,
\end{equation}
where $\normal$ denotes the unit normal vector field on $\partial\Omega$ pointing out of $\Omega$.

\subsection{Jump liftings and discrete gradient}

The jumps of smooth enough functions can be lifted to polynomial functions defined over $\Omega$.
Specifically, given an integer $k\ge 0$, for each $F\in\Fh$ we define the local trace lifting $\rF:L^1(F)\to\Poly{k}(\Th)^d$ such that, for all $\psi \in L^1(F)$,
\begin{equation}\label{eq:rF}
  \int_\Omega\rF\psi\cdot\tau_h
  = \int_F\psi~\avg{\tau_h}\cdot\normal_F
  \qquad\forall\tau_h\in\Poly{k}(\Th)^d
\end{equation}  
and we let $\Rh: W^{1,1}(\Th)\to\Poly{k}(\Th)^d$ be the global face jumps lifting such that, for any $\varphi \in W^{1,1}(\Th)$,
\begin{equation}\label{eq:Rh}
  \Rh\varphi\coloneq\sum_{F\in\Fh}\rF(\jump{\varphi}).
\end{equation}
Finally, we define the discrete gradient $\Gh: W^{1,1}(\Th)\to L^1(\Omega)^d$ setting
\begin{equation}\label{eq:Gh}
  \Gh\varphi\coloneq\nabla_h\varphi - \Rh\varphi,
\end{equation}
where $\nabla_h$ denotes the broken gradient on $\Th$.

For any $p \in [1,\infty)$, we define the following broken norm:
For all $\varphi\in W^{1,p}(\Th)$,
\begin{equation}\label{eq:norm.1.p.h}
  \text{
    $\norm{1,p,h}{\varphi}
    \coloneq\left(
    \norm{L^p(\Omega)^d}{\nabla_h\varphi}^p
    + \seminorm{1,p,h}{\varphi}^p
    \right)^{\frac1p}$
    with $\seminorm{1,p,h}{\varphi}\coloneq\left(
    \sum_{F\in\Fh}h_F^{1-p}\norm{L^p(F)}{\jump{\varphi}}^p
    \right)^{\frac1p}$,
  }
\end{equation}
  which extends as follows to the case $p = \infty$:
  \begin{equation}\label{eq:norm.1.infty.h}
    \text{%
      $\norm{1,\infty,h}{\varphi}
      \coloneq \norm{L^{\infty}(\Omega)^d}{\nabla_h\varphi}  
      + \seminorm{1,\infty,h}{\varphi}$
      with $\seminorm{1,\infty,h}{\varphi}\coloneq
      \max_{F\in\Fh} h_F^{-1}\norm{L^{\infty}(F)}{\jump{\varphi}}$.
    }
  \end{equation}
In local estimates, we will also need the following local versions 
of the norms \eqref{eq:norm.1.p.h} and \eqref{eq:norm.1.infty.h}:
For all $T\in\Th$ and all $\varphi\in W^{1,p}(\elements{T})$,
\begin{equation}\label{eq:norm.1.p.T}
  \text{
    $\norm{1,p,T}{\varphi}
    \coloneq\left(
    \norm{L^p(T)^d}{\nabla_h\varphi}^p
    + \seminorm{1,p,T}{\varphi}^p
    \right)^{\frac1p}$
    with $\seminorm{1,p,T}{\varphi}
    \coloneq\left(
    \sum_{F\in\FT}h_F^{1-p}\norm{L^p(F)}{\jump{\varphi}}^p
    \right)^{\frac1p}$
  }
\end{equation}
  and
\[
  \text{
    $\norm{1,\infty,T}{\varphi}
    \coloneq
    \norm{L^\infty(T)^d}{\nabla_h\varphi}
    + \seminorm{1,\infty,T}{\varphi}$
    with $\seminorm{1,\infty,T}{\varphi}
    \coloneq
    \max_{F\in\FT} h_F^{-1}\norm{L^{\infty}(F)}{\jump{\varphi}}$.
  }
\]
It is easy to check that, for all $\varphi\in W^{1,p}(\Th)$, $p \in [1,\infty]$,
\[ 
  \text{
    $\norm{1,p,h}{\varphi}^p
    \simeq\sum_{T\in\Th}\norm{1,p,T}{\varphi}^p$
    \quad and \quad 
    $\seminorm{1,p,h}{\varphi}^p
    \simeq\sum_{T\in\Th}\seminorm{1,p,T}{\varphi}^p$.
  }
\] 

\begin{lemma}[Properties of the jump lifting]
  It holds, for any integer $k\ge 0$ and any $p \in [1,\infty]$:
  \begin{enumerate}
  \item \emph{Boundedness.} For all $\varphi\in W^{1,p}(\Th)$, it holds 
    \begin{equation}\label{eq:rF.le.jump}
      \norm{L^p(\Omega)^d}{\rF(\jump{\varphi})}
      \lesssim h_F^{\frac{1-p}{p}}\norm{L^p(F)}{\jump{\varphi}}
      \qquad\forall F\in\Fh
    \end{equation}
    with the convention that $h_F^{\frac{1-\infty}{\infty}} \coloneq h_F^{-1}$ and 
    \begin{equation} \label{eq:Rh:boundedness}
      \norm{L^p(T)^d}{\Rh\varphi}
      \lesssim\seminorm{1,p,T}{\varphi}
      \qquad\forall T\in\Th.
    \end{equation}
  \item \emph{Approximation.} For any $w\in W_0^{1,p}(\Omega)$ (with $W_0^{1,p}(\Omega)$ denoting the closure of $C_{\rm c}^\infty(\Omega)$ in $W^{1,p}(\Omega)$) such that $w\in W^{r+1,p}(\Th)$ for some $r\in\{0,\ldots,k\}$,
    \begin{equation} \label{eq:Rh:approximation}
      \norm{L^p(T)^d}{\Rh\lproj{h}{k}w}
      \lesssim h_T^r\seminorm{W^{r+1,p}(\elements{T})}{w}
      \qquad\forall T\in\Th.
    \end{equation}
  \end{enumerate}
\end{lemma}

\begin{proof}
  \underline{\emph{Proof of \eqref{eq:rF.le.jump}--\eqref{eq:Rh:boundedness}.}}
  It holds, for all $F\in\Fh$,
  \[
  \begin{aligned}
    \norm{L^p(\Omega)^d}{\rF(\jump{\varphi})}
    &= \sup_{\tau\in L^{p'}(\Omega)^d \setminus \{0\}}\frac{\int_\Omega\rF(\jump{\varphi})\cdot\tau}{\norm{L^{p'}(\Omega)^d}{\tau}}
    \\
    &= \sup_{\tau\in L^{p'}(\Omega)^d \setminus \{0\}}\frac{\int_\Omega\rF(\jump{\varphi})\cdot\lproj{h}{k}\tau}{\norm{L^{p'}(\Omega)^d}{\tau}}
    \overset{\eqref{eq:rF}}{=}
    \sup_{\tau\in L^{p'}(\Omega)^d \setminus \{0\}}\frac{\int_F\jump{\varphi}\avg{\lproj{h}{k}\tau}\cdot\normal_F}{\norm{L^{p'}(\Omega)^d}{\tau}},
  \end{aligned}
  \]
  where the introduction of the $L^2$-orthogonal projector $\lproj{h}{k}$ in the second equality is made possible by its definition.
  We next write, setting $h^{-\frac{1}{p'}} \coloneq 1$ if $p' = \infty$,
  \[
  \left|\int_F\jump{\varphi}\avg{\lproj{h}{k}\tau}\cdot\normal_F\right|
  \lesssim h_F^{-\frac{1}{p'}}\norm{L^p(F)}{\jump{\varphi}}\norm{L^{p'}(\elements{F})^d}{\lproj{h}{k}\tau}
  \lesssim h_F^{-\frac{1}{p'}}\norm{L^p(F)}{\jump{\varphi}}\norm{L^{p'}(\elements{F})^d}{\tau},
  \]
  where we have used a H\"older inequality with exponents $(p,p',\infty)$ along with the fact that $\norm{L^\infty(F)^d}{\normal_F}\le 1$ followed by the discrete trace inequality \cite[Lemma~1.32]{Di-Pietro.Droniou:20} in the first passage,
  while the second passage is a consequence of the $L^{p'}$-boundedness of the $L^2$-orthogonal projector (cf. \cite[Lemma~1.44]{Di-Pietro.Droniou:20}).
  Additionally noticing that, by \eqref{eq:p'}, $-\frac{1}{p'} = \frac{1-p}{p}$ and that $\norm{L^{p'}(\elements{F})^d}{\tau}\le\norm{L^{p'}(\Omega)^d}{\tau}$, yields \eqref{eq:rF.le.jump}.
  
Let now $T\in\Th$. 
In order to estimate $\norm{L^p(T)^d}{\Rh\varphi}$, we first recall that, for any $F \in \Fh$, the support of $\rF(\jump{\varphi})$ is $\omega_F$, then use a triangle inequality together with \eqref{eq:rF.le.jump}, and finally use $\card(\FT) \lesssim 1$:
$$
\norm{L^p(T)^d}{\Rh\varphi}
= \Norm{L^p(T)^d}{ \sum_{F\in\FT} \rF(\jump{\varphi})}
\lesssim \sum_{F\in\FT} h_F^{\frac{1-p}{p}}\norm{L^p(F)}{\jump{\varphi}}
\lesssim \seminorm{1,p,T}{\varphi} \, ,
$$
which is the bound \eqref{eq:Rh:boundedness}.
\medskip\\
  \underline{\emph{Proof of \eqref{eq:Rh:approximation}.}}
  If $p \in [1,\infty)$, using the result proved in the previous point, we can write, for all $T\in\Th$,
  \[
  \begin{aligned}
    \norm{L^p(T)^d}{\Rh\lproj{h}{k}w}^p
    &\lesssim\sum_{F\in\FT} h_F^{1-p} \norm{L^p(F)}{\jump{\lproj{h}{k}w}}^p
    \\
    &= \sum_{F\in\FT} h_F^{1-p} \norm{L^p(F)}{\jump{\lproj{h}{k}w - w}}^p
    \lesssim h_T^{pr}\sum_{F\in\FT}\seminorm{W^{r+1,p}(\elements{F})}{w}^p,  
  \end{aligned}
  \]
  where, to insert $w$ in the second passage, we have used the fact that its jumps vanish across interfaces and its trace on $\partial\Omega$ is zero,
  while the conclusion follows from scaled trace inequalities and approximation properties of the $L^2$-orthogonal projector, additionally recalling that $h_T \lesssim h_{T'}$ for all $T' \in \TT$ by mesh regularity.
  Using $\card(\FT) \lesssim 1$, \eqref{eq:Rh:approximation} follows.
    If $p = \infty$, we have
    \[
    \norm{L^\infty(T)^d}{\Rh\lproj{h}{k}w}
    \lesssim \sum_{F \in \FT} h_F^{-1} \norm{L^\infty(F)}{\jump{\lproj{h}{k}w}}
    = \sum_{F \in \FT} h_F^{-1} \norm{L^\infty(F)}{\jump{\lproj{h}{k}w - w}}
    \lesssim h_T^r \seminorm{W^{r+1,\infty}(\TT)}{w},
    \]
    which concludes the proof.
\end{proof}

\begin{lemma}[Approximation properties of the discrete gradient]
  For all integer $k\ge 0$, all $p\in[1,\infty]$, and all $w\in W_0^{1,p}(\Omega)\cap W^{r+1,p}(\Th)$ with $r\in\{0,\ldots,k\}$, it holds
  \begin{equation} \label{eq:Gh:approximation}
    \norm{L^p(T)^d}{\Gh\lproj{h}{k}w - \nabla w}
    \lesssim h_T^r\seminorm{W^{r+1,p}(\TT)}{w}
    \qquad\forall T\in\Th.
  \end{equation}
\end{lemma}

\begin{proof}
  Using  \eqref{eq:Gh} and a triangle inequality, we obtain
  \[
  \norm{L^p(T)^d}{\Gh\lproj{h}{k}w - \nabla w}
  \le\norm{L^p(T)^d}{\nabla\lproj{T}{k} w - \nabla w}
  + \norm{L^p(T)^d}{\Rh\lproj{h}{k}w}.
  \]
  The conclusion follows using the approximation properties of the $L^2$-orthogonal projector for the first term and \eqref{eq:Rh:approximation} for the second.
\end{proof}

\begin{remark}[Local boundedness of $\Gh\circ\lproj{h}{k}$]
  For any $q\in[1,\infty]$, combining a triangle inequality with \eqref{eq:Gh:approximation} written for $r = 0$, it is readily inferred that, for all $\varphi\in W^{1,q}(\Th)$,
  \begin{equation} \label{eq:Gh.lproj:boundedness}
    \norm{L^q(T)^d}{\Gh\lproj{h}{k}\varphi}
    \lesssim \seminorm{W^{1,q}(\TT)}{\varphi}
    \qquad\forall T\in\Th.
  \end{equation}
\end{remark}


\section{Discrete problem and main results}\label{sec:discr}

\subsection{Discrete problem}

From this point on, we let a Sobolev exponent $p\in(1,\infty)$ and polynomial degree $k\ge 1$ be fixed.
The diffusion term is discretized, similarly to what is proposed in \cite{Burman.Ern:08}, by the function $a_h:\Poly{k}(\Th)\times\Poly{k}(\Th)\to\Real$ such that, for all $(w_h,v_h)\in\Poly{k}(\Th)\times\Poly{k}(\Th)$,
\begin{equation}\label{eq:ah}
  a_h(w_h,v_h)
  \coloneq\int_\Omega\sigma(\Gh w_h)\cdot\Gh v_h
  + s_h(w_h,v_h),
\end{equation}
where
\[
  s_h(w_h,v_h)
  \coloneq\sum_{F\in\Fh} h_F^{1-p} \int_F \sigma_1(\jump{w_h})\jump{v_h}
  = \sum_{F\in\Fh}h_F^{1-p}\int_F|\jump{w_h}|^{p-2}\jump{w_h}\jump{v_h}.
\]

The discretization of the advection-reaction terms hinges on the bilinear form $b_h:\Poly{k}(\Th)\times\Poly{k}(\Th)\to\Real$ such that, for all $(w_h,v_h)\in\Poly{k}(\Th)\times\Poly{k}(\Th)$,
\begin{equation}\label{eq:bh}
  \begin{aligned}
    b_h(w_h,v_h)
    & =
    -\int_\Omega w_h(\beta\cdot\nabla_h v_h) 
    + \int_\Omega\mu w_h v_h
    + \sum_{F\in\Fh}\int_F(\beta\cdot\normal_F) \avg{w_h} \jump{v_h}  \\
    & \quad
    + \frac12\sum_{F\in\Fh}\vref{F} \int_F \jump{w_h}\jump{v_h},
  \end{aligned} 
\end{equation}
where, for all $F\in\Fh$, we have introduced the face reference velocity
\[
  \vref{F} \coloneq \norm{L^\infty(F)}{\beta\cdot\normal_F}.
\]

Notice that the stabilization term is not the classical upwind, but rather a stronger version based on the reinterpretation as jump penalty provided in \cite{Brezzi.Marini.ea:04}.

\begin{remark}[Generalizations]
  The bilinear form $b_h$ includes suitable terms that will be used to control the diffusive and advection terms on advection-dominated faces. The above formulation (and, in many cases, also the theoretical results that follow) could be easily extended to other choices, such as including cross-wind or making $\vref{F}$ dependent on some computable estimate of the local P\'eclet number.
In particular, one could switch to standard upwind stabilization on boundary faces to correctly treat boundary conditions in the vanishing diffusion case.
\end{remark}
  
The discrete problem reads:
Find $u_h\in\Poly{k}(\Th)$ such that
\begin{equation}\label{eq:discrete}
  a_h(u_h,v_h) + b_h(u_h,v_h) = \int_\Omega f v_h
  \qquad\forall v_h\in\Poly{k}(\Th).
\end{equation}

\subsection{Main results}\label{sec:main}

In this section we collect the main results of the analysis of problem \eqref{eq:discrete}.
The error estimate accounts for the different regimes in each mesh element/face, as identified by local P\'eclet numbers (for a similar local approach in a different context, see, for instance, \cite{Di-Pietro.Droniou:23}).

\subsubsection{Dimensionless numbers and reference quantities }

In order to state these convergence results, we need to define here key reference quantities and dimensionless numbers.
For any function $w \in W^{1,p}(\Omega)$ and any mesh element $T\in\Th$, we define the element Péclet number as follows.
If $\vref{T}\coloneq\norm{L^\infty(T)^d}{\beta}$ vanishes, we set $\Pe_T(w)=0$; otherwise,
\begin{equation}\label{eq:PeT}
  \Pe_T(w)\coloneq\frac{\vref{T}h_T}{\dref{T}(w)}
  \quad\text{with}\quad
  \dref{T}(w)\coloneq\norm{L^\infty(\elements{T})}{|\nabla w|^{p-2}} \, ,
\end{equation}
with the convention that $\dref{T}(w) = +\infty$ (and thus $\Pe_T(w)=0$) if the restriction of $|\nabla w|^{p-2}$ is not in $L^\infty(\omega_T)$.
Furthermore, we define the reference time:
\begin{equation}\label{eq:tref}
  \tref\coloneq\frac{1}{\max(\norm{L^\infty(T)}{\mu},\seminorm{W^{1,\infty}(T)^d}{\beta})}.
\end{equation}
Similarly, for any $F\in\Fh$, we define the face Péclet number as follows.
If $\vref{F}=0$, we set $\Pe_F(w)=0$; otherwise
\begin{equation}\label{eq:PeF.drefF}
  \Pe_F(w)\coloneq\frac{\vref{F} h_F}{\dref{F}(w)}\quad\text{with}\quad
  \dref{F}(w)\coloneq
  \max\left(
  \norm{L^\infty(F)}{|\nabla w|^{p-2}},
  h_F^{2-p}\norm{L^\infty(F)}{|\jump{\lproj{h}{k}w}|^{p-2}}
  \right) \, ,
\end{equation}
where again $\dref{F}(w) = +\infty$ (and thus $\Pe_F(w)=0$) whenever the involved functions are not in $L^\infty(F)$. 

Notice that the face P\'eclet number accounts for the fact that the stabilization term introduces additional numerical diffusion. In practical situations, this numerical diffusion can be expected to be small compared to the physical one.

We partition the sets of mesh elements and faces based on the values of the local P\'eclet numbers.
Specifically, given a smooth enough function $w:\Omega\to\Real$, we set
\[ 
  \begin{aligned}
    \Tha(w) &\coloneq \left\{
    T\in\Th\st \Pe_T(w) > 1
    \right\},
    &\qquad
    \Thd(w) &\coloneq \Th \setminus \Tha(w),
    \\
    \Fha(w) &\coloneq \left\{
    F\in\Fh\st \Pe_F(w) > 1
    \right\},
    &\qquad
    \Fhd(w) &\coloneq \Fh \setminus \Fha(w).
  \end{aligned}
\] 


\subsubsection{Norms}

The relevant norm for the analysis of the diffusion terms is $\norm{1,p,h}{{\cdot}}$ (cf. \eqref{eq:norm.1.p.h}) as well its restriction to an element $T\in\Th$ (cf. \eqref{eq:norm.1.p.T}).
The norm for the advective and reactive terms is, on the other hand, given by
\begin{equation}\label{eq:norm.beta.mu.h}
  \norm{\beta,\mu,h}{v_h}
  \coloneq\left(
  \frac12\sum_{F\in\Fh}\vref{F}\norm{L^2(F)}{\jump{v_h}}^2
  + \norm{L^2(\Omega)}{\mu^{\frac12}v_h}^2
  \right)^{\frac12}
  \qquad\forall v_h\in\Poly{k}(\Th).
\end{equation}
This choice of advection-reaction norm is justified as follows. 
By standard arguments (which essentially amount to applying the integration by parts formula \eqref{eq:ibp} with $(\tau,v) = (\beta w_h,v_h)$ to the first term in the right-hand side of \eqref{eq:bh}, using the continuity of $\beta\cdot\normal_F$ across interfaces, and recalling that $\nabla\cdot\beta = 0$), it is easy to check that
\begin{equation}\label{eq:bh:coercivity}
  b_h(v_h,v_h) = \norm{\beta,\mu,h}{v_h}^2
  \qquad\forall v_h\in\Poly{k}(\Th),
\end{equation}
showing that $b_h$ is coercive with respect to the norm defined by \eqref{eq:norm.beta.mu.h} with coercivity constant equal to 1.

\subsubsection{Error estimate}

The following theorem contains an estimate of the error between the solution of the discrete problem \eqref{eq:discrete} and the projection of the continuous solution that tracks the dependence of the convergence rate on the local regime.
We remark that the regularity conditions required below for $u$ are implied, for instance, by the simpler but less sharp requirement $w \in W^{1,p}(\Omega) \cap W^{r+1,\overline{p}}(\Th)$ with $\overline{p}=\max{\{2,2p-2,p'\}}$.

\begin{theorem}[Convergence]\label{thm:convergence}
  Denote, respectively, by $u \in W^{1,p}(\Omega)$ and by $u_h \in \Poly{k}(\Th)$ the solutions of the weak formulation of problem \eqref{eq:strong} and of the discrete problem \eqref{eq:discrete}.
  Additionally assume that, for some $r\in\{0,\ldots,k\}$,
  \begin{itemize}
  \item $u_{|\omega_T} \in W^{r+1,p}(\TT)$ for all $T\in\Thd(u)$;
  \item $u_{|T} \in H^{r+1}(T)$ for all $T\in\Th$;
  \item $u_{|\omega_F} \in W^{r+1,p}(\TF)\cap W^{r+1,p'}(\TF)$ and $\sigma(\nabla u)_{|\omega_F} \in W^{r,p'}(\TF)^d$ for all $F\in\Fhd(u)$;
  \item $u_{|\omega_F} \in H^{r+1}(\TF)$ and $\sigma(\nabla u)_{|\omega_F} \in H^{r+\frac12}(\TF)^d$ for all $F\in\Fha(u)$.
  \end{itemize}
  Then, letting
      \begin{equation}\label{eq:q}
      q\coloneq\begin{cases}
      2 & \text{if $p<2$}, \\
      p & \text{if $p\ge 2$},
      \end{cases}
      \end{equation}
      it holds
  \begin{equation}\label{eq:err.est}
    \begin{aligned}
      &\norm{1,p,h}{u_h - \lproj{h}{k}u}^q
      + \norm{\beta,\mu,h}{u_h - \lproj{h}{k}u}^2
      \\
      &\quad
      \lesssim
      \sum_{T\in\Th}\tref^{-2}\underline{\mu}_T^{-1}h_T^{2(r+1)}\seminorm{H^{r+1}(T)}{u}^2
      \\
      &\qquad
      + \sum_{T\in\Tha(u)} \vref{T} h_T^{2r+1} \seminorm{H^{r+1}(T)}{u}^2
      +  \sum_{T\in\Thd(u)}\begin{cases}
        h_T^{rp}\seminorm{W^{r+1,p}(\TT)}{u}^{p}
        & \text{if $p< 2$}
        \\
        h_T^{2r}\seminorm{W^{r+1,p}(\TT)}{u}^2
        & \text{if $p\ge 2$}
      \end{cases}
      \\
      &\qquad
      + \sum_{F\in\Fha(u)} h_F^{2r+1}
      \left(
      \dref{F}(u)^{-1} \seminorm{H^{r+\frac12}(\elements{F})^d}{\sigma(\nabla u)}^2
      + \vref{F} \seminorm{H^{r+1}(\TF)}{u}^2
      \right)
      + \sum_{F\in\Fhd(u)} h_F^{rp} \seminorm{W^{r+1,p}(\elements{F})}{u}^p
      \\
      &\qquad
      + \left[
      \sum_{F\in\Fhd(u)} h_F^{r p'}\left(
      \seminorm{W^{r,p'}(\elements{F})^d}{\sigma(\nabla u)}^{p'}
      + \dref{F}(u)^{p'} \seminorm{W^{r+1,p'}(\TF)}{u}^{p'}
      \right)
      \right]^{\frac{q'}{p'}}.
    \end{aligned}
   \end{equation}
\end{theorem}

\begin{proof}
  See Section \ref{sec:err.est}.
\end{proof}

The above convergence result is fully local, being able to deliver sharp estimates also in situations where diffusion or advection dominate in different areas of the domain. This feature is particularly important in the present nonlinear situation, where the distinction among the two cases depends on the solution itself and not only on some data given a priori. Notice that, for the sake of conciseness, here we do not consider the trivial case of dominating reaction. 

For the more interesting case $p < 2$, the above estimates are ``optimal'' in the sense that, for regular solutions, the bound yields the same asymptotic order of convergence as for conforming Finite Element (FE) schemes, i.e., $\mathcal{O}(h^{\frac{rp}{2}})$ \cite{Barrett.Liu:93,Hirn:13}.
Furthermore, in the pre-asymptotic regime, our estimate underlines a better error reduction rate
in the areas of the domain where advection dominates (behaving as $h^{r+\frac12}$). In this respect, note that the negative power of $\dref{F}$ appearing in the bound above is balanced by the associated $\sigma$ term, see Remark \ref{newrem-1}.
The case $p = 2$ corresponds to a linear diffusion-advection-reaction problem, for which classical estimates are recovered (see, e.g., \cite{Di-Pietro.Ern.ea:08,Ayuso.Marini:09} and also \cite[Section 4.6]{Di-Pietro.Ern:12}).
In the case $p>2$, the same observations apply, except for the fact that the asymptotic convergence rate now compares unfavorably to the conforming FE case, due to the presence of an $\mathcal{O}(h^{rp'})$ term in the right hand side (to be compared with $\mathcal{O}(h^{2r})$).
This aspect could be possibly improved by introducing a stronger jump term $s_h$ (which, on the other, hand would lead to a weaker pre-asymptotic reduction rate in advection dominated regimes) or by introducing some suitable tweaks in the analysis, see Remark \ref{newrem-2}.


\section{Theoretical analysis}\label{sec:ana}

\subsection{Properties of the diffusion function}

\begin{lemma}[Stability of $a_h$]\label{eq:ah:stability}
  For any $w_h,v_h\in\Poly{k}(\Th)$, recalling the definition \eqref{eq:q} of $q$ and assuming that $\norm{1,p,h}{w_h} + \norm{1,p,h}{v_h}\lesssim 1$ if $p<2$, there exists $C_a$ independent of $h$ (but possibly depending on $\Omega$, $p$, and the mesh regularity parameter) such that
  \begin{equation}\label{eq:ah:stability:3}
    C_a \norm{1,p,h}{w_h - v_h}^q
    \lesssim  a_h(w_h,w_h - v_h) - a_h(v_h,w_h - v_h).
  \end{equation}
\end{lemma}

\begin{proof}
  The proof is a straightforward adaptation of the monotonicity properties of $\sigma$ and the arguments of \cite[Point (ii) of Theorem~6.19]{Di-Pietro.Droniou:20}.
\end{proof}

We start by estimating the error stemming from the diffusion term.

\begin{lemma}[Estimate of the discrete diffusion error]\label{eq:ah:consistency}
  Let $w\in W^{1,p}(\Omega)$ be such that
  $\sigma(\nabla w)\in W^{1,p'}(\Th)^d$ and $\nabla \cdot \sigma(\nabla w)\in L^{p'}(\Omega)$. Let's define the diffusion error linear form $\Erra:\Poly{k}(\Th)\to\Real$ such that, for all $v_h\in\Poly{k}(\Th)$,
  \begin{equation}\label{eq:Erra}
    \Erra(w;v_h)
    \coloneq -\int_\Omega\nabla\cdot\sigma(\nabla w)~v_h
    - a_h(\lproj{h}{k}w,v_h).
  \end{equation}
  Additionally assume that, for some $r\in\{0,\ldots,k\}$,
  \begin{itemize}
  \item $w_{|\omega_T} \in W^{r+1,p}(\TT)$ for all $T\in\Thd(w)$;
  \item $w_{|T} \in H^{r+1}(T)$ for all $T\in\Tha(w)$;
  \item $w_{|\omega_F} \in W^{r+1,p}(\TF)$ and $\sigma(\nabla w)_{|\omega_F} \in W^{r,p'}(\TF)^d$ for all $F\in\Fhd(w)$;
  \item $w_{|\omega_F} \in H^{r+1}(\TF)$ and $\sigma(\nabla w)_{|\omega_F} \in H^{r+\frac12}(\TF)^d$ for all $F\in\Fha(w)$.
  \end{itemize}
  Then, recalling \eqref{eq:q}, it holds, for any $w_h \in \Poly{k}(\Th)$ and any real number $\delta > 0$,
  \begin{equation}\label{eq:Erra:estimate}
    \begin{aligned}
      &\Erra(w;w_h-\lproj{h}{k}w)
      \\
      &\quad
      \le
      \delta\left(
      a_h(w_h, w_h - \lproj{h}{k} w)
      - a_h(\lproj{h}{k} w, w_h - \lproj{h}{k} w)
      + \seminorm{1,p,h}{w_h - \lproj{h}{k}w}^q
      + \norm{\beta,\mu,h}{w_h - \lproj{h}{k}w}^2
      \right)
      \\
      &\qquad
      + c(\delta)\left(
      \sum_{T\in\Tha(w)} \vref{T} h_T^{2r+1} \seminorm{H^{r+1}(T)}{w}^2
      +  \sum_{T\in\Thd(w)}\begin{cases}
        h_T^{rp}\seminorm{W^{r+1,p}(\TT)}{w}^{p}
        & \text{if $p< 2$}
        \\
        h_T^{2r}\seminorm{W^{r+1,p}(\TT)}{w}^2
        & \text{if $p\ge 2$}
      \end{cases}
      \right)
      \\
      &\qquad
      + c(\delta)
      \sum_{F\in\Fha(w)}\dref{F}^{-1}(w) h_F^{2r+1} \seminorm{H^{r+\frac12}(\elements{F})^d}{\sigma(\nabla w)}^2
      + c(\delta)\left(
      \sum_{F\in\Fhd(w)} h_F^{r p'}\seminorm{W^{r,p'}(\elements{F})}{\sigma(\nabla w)}^{p'}
      \right)^{\frac{q'}{p'}}
      \\
      &\qquad
      + c(\delta)\left(
      \sum_{F\in\Fha(w)} \vref{F} h_F^{2r+1} \seminorm{H^{r+1}(\TF)}{w}^2
      + \sum_{F\in\Fhd(w)} h_F^{rp} \seminorm{W^{r+1,p}(\elements{F})}{w}^p 
      \right),
    \end{aligned}
  \end{equation}
  with $c(\delta)$ independent of the particular mesh in $\{ \Th \}_h$ and the function $w$. 
\end{lemma}

\begin{proof}
  Let, for the sake of brevity, 
  \[
  v_h \coloneq w_h - \lproj{h}{k}w.
  \]
  Using the integration by parts formula \eqref{eq:ibp} for the first term in the right-hand side of \eqref{eq:Erra} along with the fact that $\jump{\sigma(\nabla w)}\cdot\normal_F$ vanishes for all $F\in\Fhi$ (which expresses the continuity of normal fluxes),
  expanding $a_h$ according to its definition \eqref{eq:ah},
  adding
  $0 = \int_\Omega\lproj{h}{k}\sigma(\nabla w)\cdot\Rh v_h - \sum_{F\in\Fh}\int_F\avg{\lproj{h}{k}\sigma(\nabla w)}\cdot\normal_F\,\jump{v_h}$ (cf. \eqref{eq:Rh} and \eqref{eq:rF}),
  and adding and subtracting $\int_\Omega\sigma(\nabla w)\cdot\Rh v_h$,
  we arrive at the following decomposition of the error:
  \begin{equation}\label{eq:Erra:basic}
    \begin{aligned}
      &\Erra(w;v_h)
      \\
      &\quad= \underbrace{
        \vphantom{\sum_{F\in\Fh}}
        \int_\Omega        
          \left[
            \sigma(\Gh\lproj{h}{k}w) - \sigma(\nabla w)
            \right]\cdot\Gh(\lproj{h}{k}w - w_h)        
      }_{\term_1}
      -\underbrace{
        \sum_{F\in\Fh}\int_F\avg{\sigma(\nabla w) - \lproj{h}{k}\sigma(\nabla w)}\cdot\normal_F~\jump{v_h}
      }_{\term_2}
      \\
      &\qquad
      + \cancel{\int_\Omega\left[
          \sigma(\nabla w) - \lproj{h}{k}\sigma(\nabla w)
          \right]\cdot\Rh v_h}
      + \underbrace{%
        s_h(\lproj{h}{k}w, \lproj{h}{k}w - w_h),
      }_{\term_3}
    \end{aligned}
  \end{equation}
  where the cancellation follows from the definition of $\lproj{h}{k}$ after recalling that $\Rh v_h\in\Poly{k}(\Th)^d$.
  We next proceed to estimate the other terms in the right-hand side.
  \medskip\\
  \underline{\textbf{Estimate of $\term_1$}.}
  For the first term, we start by writing $\term_1 = \sum_{T\in\Th}\term_1(T)$ and consider a single $T \in \Th$.
  Using the bound \eqref{hirn-bound} with $n = d$ and $(x,y,z) = (\Gh\lproj{h}{k} w, \Gh w_h, \nabla w)$ and recalling that $v_h = w_h - \lproj{h}{k}w$, we obtain
  \begin{equation}\label{eq:T1:initial}
    \begin{aligned}
      \term_1(T)
      &\le
      \delta\left( \int_T\sigma(\Gh w_h) \cdot \Gh v_h - \int_T\sigma(\Gh\lproj{h}{k}w) \cdot \Gh v_h \right)
      + c(\delta) \term_{1,\err}(T),
    \end{aligned}
  \end{equation}
  where
  \[
  \term_{1,\err}(T)
  \coloneq
  \int_{T} (|\nabla w|+|\Gh\lproj{h}{k}w|)^{p-2} |\nabla w - \Gh\lproj{h}{k}w|^2.
  \]  
  We now distinguish between diffusion-dominated and advection-dominated elements of the mesh to estimate $\term_{1,\err}(T)$.
  
  Let first $T \in \Thd(w)$.
  In the case $p < 2$, we use the fact that $|\nabla w - \Gh\lproj{h}{k}w| \le |\nabla w| + |\Gh\lproj{h}{k}w|$ almost everywhere in $T$ along with the fact that $\Real^+\ni x\mapsto x^{p-2}\in\Real$ is strictly decreasing to write
  \begin{equation}\label{eq:Erra:T1err:diffusive.p<2}
    \term_{1,\err}(T)
    \le  \int_{T} |\nabla w - \Gh\lproj{h}{k}w|^p
    = \norm{L^p(T)^d}{\nabla w - \Gh\lproj{h}{k}w}^p
    \overset{\eqref{eq:Gh:approximation}}\lesssim  h_T^{rp} \seminorm{W^{r+1,p}(\TT)}{w}^p.
  \end{equation}
  In the case  $p \ge 2$, on the other hand, we apply a H\"older inequality with exponents $\left(\frac{p}{p-2},\frac{p}{2}\right)$
  and a triangle inequality to write 
  \begin{equation}\label{eq:Erra:T1err:diffusive.p>=2}
    \begin{aligned}
      \term_{1,\err}(T)
      &\lesssim
      \left(
      \norm{L^p(T)^d}{\nabla w}
      + \norm{L^p(T)^d}{\Gh \lproj{h}{k} w}
      \right)^{p-2}\norm{L^p(T)^d}{\nabla w - \Gh \lproj{h}{k} w}^2
      \\
      \overset{\eqref{eq:Gh.lproj:boundedness},\,\eqref{eq:Gh:approximation}}&\lesssim
      \norm{L^p(\TT)^d}{\nabla w}^{p-2}
      h_T^{2r}\seminorm{W^{r+1,p}(\TT)}{w}^2
      \lesssim  h_T^{2r}\seminorm{W^{r+1,p}(\TT)}{w}^2,
    \end{aligned}
  \end{equation}
  where the conclusion follows from the assumption $\norm{L^p(\Omega)^d}{\nabla w}\lesssim 1$.

  Let now $T\in\Tha(w)$. We first consider the case $p < 2$.
  Using again the fact that $\Real^+\ni x\mapsto x^{p-2}\in\Real$ is stricly decreasing,
  then applying a H\"older inequality with exponents $(\infty,1)$ and using the approximation properties \eqref{eq:Gh:approximation} of $\Gh\circ\lproj{h}{k}$,
  and finally recalling that $\Pe_T(w) > 1$ (cf. \eqref{eq:PeT} for its definition), we have
  \begin{equation}\label{eq:Erra:T1err:advective.p<2}
    \begin{aligned}
      \term_{1,\err}(T)
      &\lesssim  \int_{T} |\nabla w|^{p-2} |\nabla w - \Gh\lproj{h}{k}w|^2
      \lesssim  \norm{L^\infty(T)}{|\nabla w|^{p-2}}~\norm{L^2(T)^d}{\nabla w - \Gh\lproj{h}{k}w}^2
      \\
      &\lesssim  \dref{T}(w) h_T^{2r} \seminorm{H^{r+1}(T)}{w}^2
      \le  \vref{T} h_T^{2r+1} \seminorm{H^{r+1}(T)}{w}^2.
    \end{aligned}
  \end{equation}
  In the case $p\ge 2$, on the other hand, the local boundedness \eqref{eq:Gh.lproj:boundedness} of $\Gh\circ\lproj{h}{k}$ with $q=\infty$ along with the definition \eqref{eq:PeT} of $\dref{T}(w)$ easily leads to
  \begin{equation}\label{eq:Erra:T1err:advective.p>=2}
    \term_{1,\err}(T)
    \lesssim \dref{T}(w) h_T^{2r} \seminorm{H^{r+1}(T)}{w}^2
    \lesssim \vref{T} h_T^{2r+1}\seminorm{H^{r+1}(T)}{w}^2,
  \end{equation}
  where the conclusion follows again using $\Pe_T(w) > 1$.

  Plugging the estimates
  \eqref{eq:Erra:T1err:diffusive.p<2},
  \eqref{eq:Erra:T1err:diffusive.p>=2},
  \eqref{eq:Erra:T1err:advective.p<2},
  and \eqref{eq:Erra:T1err:advective.p>=2} into \eqref{eq:T1:initial}, we arrive at
  \begin{equation}\label{eq:Erra:estimate:T1}
    \begin{aligned}
      \term_1
      &\le
      \delta\left(
      \int_\Omega\sigma(\Gh w_h)\cdot\Gh v_h
      - \int_\Omega\sigma(\Gh\lproj{h}{k}w)\cdot\Gh v_h
      \right)
      \\
      &\quad
      + c(\delta) \sum_{T\in\Tha(w)} \vref{T} h_T^{2r+1} \seminorm{H^{r+1}(T)}{w}^2
      + c(\delta) \sum_{T\in\Thd(w)}\begin{cases}
        h_T^{rp}\seminorm{W^{r+1,p}(\TT)}{w}^{p}
        & \text{if $p< 2$},
        \\
        h_T^{2r}\seminorm{W^{r+1,p}(\TT)}{w}^2
        & \text{if $p\ge 2$}.      
      \end{cases}
    \end{aligned}
  \end{equation}
  \medskip\\
  \underline{\textbf{Estimate of $\term_2$}.}
  For the second term, we write $\term_2 = \sum_{F\in\Fh}\term_2(F)$ and, for all $F\in\Fhd(w)$, we estimate $\term_2(F)$ as follows:
  \begin{equation}\label{eq:usedinremark}
  \begin{aligned}
    \term_2(F)
    &\le\norm{L^{p'}(F)^d}{\avg{\sigma(\nabla w) - \lproj{h}{k}\sigma(\nabla w)}}
    ~\norm{L^p(F)}{\jump{v_h}}
    \\
    &\lesssim
    h_F^{r-\frac{1}{p'}+\frac{p-1}{p}}\seminorm{W^{r,p'}(\elements{F})}{\sigma(\nabla w)}
    ~h_F^{\frac{1-p}{p}}\norm{L^p(F)}{\jump{v_h}}
    \\
    \overset{\eqref{eq:p'}}&=
    h_F^r\seminorm{W^{r,p'}(\elements{F})}{\sigma(\nabla w)}
    ~h_F^{\frac{1-p}{p}}\norm{L^p(F)}{\jump{v_h}},
  \end{aligned}
  \end{equation}
  where we have used a triangle inequality along with the approximation properties of the $L^2$-orthogonal projector to treat the first factor in the passage to the second line.

  For $F\in\Fha(w)$, on the other hand, we first notice that $\vref{F} \neq 0$ and then
  use a Cauchy--Schwarz inequality to write
  \begin{equation}\label{eq:conv:Kminus}
  \begin{aligned}
    \term_2(F)
    &\le \vref{F}^{-\frac12} \norm{L^2(F)^d}{\avg{\sigma(\nabla w) - \lproj{h}{k}\sigma(\nabla w)}}
    ~\vref{F}^{\frac12} \norm{L^2(F)}{\jump{v_h}}
    \\
    &\lesssim \vref{F}^{-\frac12} h_F^{r} \seminorm{H^{r+\frac12}(\elements{F})^d}{\sigma(\nabla w)}
    ~\vref{F}^{\frac12} \norm{L^2(F)}{\jump{v_h}} \\
    & \lesssim \dref{F}^{-\frac12}(w) h_F^{r+\frac12} \seminorm{H^{r+\frac12}(\elements{F})^d}{\sigma(\nabla w)}
    ~\vref{F}^{\frac12} \norm{L^2(F)}{\jump{v_h}},
  \end{aligned}
  \end{equation}
  where we have used the fact that $\Pe_F(w) > 1$ to conclude.
  
  Gathering the above bounds and applying a H\"older inequality with exponents $(p',p)$ on the sum over $F\in\Fhd(w)$, a Cauchy--Schwarz inequality on the sum over $F\in\Fha(w)$,
  and using a generalized Young inequality with exponents $(q',q)$, we get
  \begin{equation}\label{eq:Erra:estimate:T2}
    \begin{aligned}
      \term_2
      &\le
      \delta \left(
      \seminorm{1,p,h}{v_h}^q
      + \norm{\beta,\mu,h}{v_h}^2
      \right)
      + c(\delta)
      \sum_{F\in\Fha(w)}\dref{F}^{-1}(w) h_F^{2r+1} \seminorm{H^{r+\frac12}(\elements{F})^d}{\sigma(\nabla w)}^2
      \\
      &\quad
      + c(\delta)\left(
      \sum_{F\in\Fhd(w)} h_F^{r p'}\seminorm{W^{r,p'}(\elements{F})}{\sigma(\nabla w)}^{p'}
      \right)^{\frac{q'}{p'}}.
    \end{aligned}
  \end{equation}
  \medskip\\
  \underline{\textbf{Estimate of $\term_3$}.}
  Finally, for the third term, we write again $\term_3 = \sum_{F\in\Fh}\term_3(F)$.
  We then first recall that $\jump{w}=0$ and then apply \eqref{hirn-bound} 
  with $n = 1$ and $(x,y,z) = (\jump{\lproj{k}{k}w}, \jump{w_h}, \jump{w})$, additionally using the fact that $v_h = w_h - \lproj{h}{k}w$; we obtain that, for all positive $\delta$,
  \[
  \begin{aligned}
    \term_3(F)
    & = h_F^{1-p}\int_F (\sigma_1(\jump{\lproj{h}{k}w}) - \sigma_1(\jump{w})) \jump{\lproj{h}{k}w - w_h}
    \\
    &\le
    \delta\left(
    h_F^{1-p}\int_F \sigma_1(\jump{w_h}) \jump{v_h}
    - h_F^{1-p}\int_F \sigma_1(\jump{\lproj{h}{k}w}) \jump{v_h}
    \right)
    + c(\delta) \term_{3,\err}(F),
  \end{aligned}
  \]
  with
  \[
  \term_{3,\err}(F) \coloneq
  h_F^{1-p} \int_F (|\jump{\lproj{h}{k}w}|+|\jump{w}|)^{p-2} |\jump{w - \lproj{h}{k}w}|^2.
  \]
  For $F\in\Fhd(w)$, this term is bounded trivially recalling that $\jump{w} = 0$:
  \[
    \term_{3,\err}(F)
    = h_F^{1-p} \int_F |\jump{\lproj{h}{k}w-w}|^p
    = h_F^{1-p} \norm{L^p(F)}{\jump{\lproj{h}{k}w-w}}^p
    \lesssim h_F^{rp} \seminorm{W^{r+1,p}(\elements{F})}{w}^p.
  \]
For $F\in\Fha(w)$, on the other hand, recalling again that $\jump{w} = 0$ and using a H\"older inequality with exponents $(\infty,1)$, we have
  \[
  \begin{aligned}
    \term_{3,\err}(F)
    &\le h_F^{1-p} \norm{L^\infty(F)}{|\jump{\lproj{h}{k}w}|^{p-2}}
      \norm{L^2(F)}{\jump{\lproj{h}{k}w-w}}^2
    \\
    \overset{\eqref{eq:PeF.drefF}}&\le \frac{\dref{F}(w)}{h_F} \norm{L^2(F)}{\jump{\lproj{h}{k}w-w}}^2
    \lesssim \vref{F} h_F^{2r+1} \seminorm{H^{r+1}(\TF)}{w}^2,
  \end{aligned}
  \]
  where the conclusion follows using the fact that $\Pe_F^{-1}(w) < 1$ for the first factor and a triangle inequality followed by the approximation properties of $\lproj{h}{k}$ for the second.

  Gathering the above estimates, we arrive at the following bound for $\term_3$:
  \begin{equation}\label{eq:Erra:estimate:T3}
  \begin{aligned}
    \term_3
    & \le
    \delta\left(
    s_h(w_h, v_h) - s_h(\lproj{h}{k}w, v_h)
    \right)
    \\
    &\quad
    + c(\delta) \sum_{F\in\Fha(w)} \vref{F} h_F^{2r+1} \seminorm{H^{r+1}(\TF)}{w}^2
    + c(\delta) \sum_{F\in\Fhd(w)} h_F^{rp} \seminorm{W^{r+1,p}(\elements{F})}{w}^p.
  \end{aligned}  
  \end{equation}
  \\
  \underline{\textbf{Conclusion.}}
  Plugging \eqref{eq:Erra:estimate:T1}, \eqref{eq:Erra:estimate:T2}, and \eqref{eq:Erra:estimate:T3} into \eqref{eq:Erra:basic} and recalling that, in each of these estimates, $\delta > 0$ is arbitrary, the conclusion follows.
\end{proof}

\begin{remark}[Negative power of $\dref{F}$]\label{newrem-1}
  As already mentioned, the negative power of $\dref{F}$ appearing in bound \eqref{eq:Erra:estimate} is balanced by the associated $\sigma$ regularity term. The $\dref{F}^{-1}$ stems from equation \eqref{eq:conv:Kminus}. The fact that the associated term $\term_2(F)$, $F\in\Fha(w)$, cannot lead to an arbitrarily large contribution to the error becomes clear by bounding such term as in \eqref{eq:usedinremark} instead of \eqref{eq:conv:Kminus} (that is, using the diffusive part of the norm instead of the advective one). Here, we decided to use \eqref{eq:conv:Kminus} in order to clearly underline the faster pre-asymptotic reduction rate occurring in advection dominated cases.
\end{remark}

\subsection{Properties of the advection-reaction bilinear form}

We now estimate the error stemming from the advection component of the equation.

\begin{lemma}[Estimate of the discrete advection-reaction error]\label{eq:bh:consistency}
  Let $w\in W^{1,p}(\Omega)$ and define the advection-reaction error linear form $\Errb(w;v_h):\Poly{k}(\Th)\to\Real$ such that, for all $v_h \in \Poly{k}(\Th)$,
  \begin{equation}\label{eq:Errb}
    \Errb(w;v_h)
    \coloneq \int_\Omega\nabla\cdot(\beta w) v_h 
    + \int_\Omega\mu w v_h
    - b_h(\lproj{h}{k}w, v_h).
  \end{equation}
  Additionally assume that $w\in H^{r+1}(\Th)$ and $w_{|\omega_F}\in W^{r+1,p'}(\TF)$ for all $F\in\Fhd(w)$ for some $r\in\{0,\ldots,k\}$.
  Then, with $q$ as in \eqref{eq:q}, it holds, for any $w_h \in \Poly{k}(\Th)$ and any real number $\delta > 0$,
  \begin{equation}\label{eq:Errb:estimate}
    \begin{aligned}
      \Errb(w;w_h - \lproj{h}{k}w)
      &\le
      \delta\left(
      \seminorm{1,p,h}{w_h - \lproj{h}{k} w}^q
      + \norm{\beta,\mu,h}{w_h - \lproj{h}{k} w}^2
      \right)
      \\
      &\quad
      + c(\delta)\left(
      \sum_{T\in\Th}\tref^{-2}\underline{\mu}_T^{-1}h_T^{2(r+1)}\seminorm{H^{r+1}(T)}{w}^2
      + \sum_{F\in\Fha(w)} \vref{F} h_F^{2r+1}\seminorm{H^{r+1}(\TF)}{w}^2
      \right)
      \\
      &\quad
      + c(\delta)\left(
      \sum_{F\in\Fhd(w)} \dref{F}(w)^{p'} h_F^{rp'}\seminorm{W^{r+1,p'}(\TF)}{w}^{p'}
      \right)^{\frac{q'}{p'}},
    \end{aligned}
  \end{equation}
  with $c(\delta)$ independent of the particular mesh in $\{ \Th \}_h$ and the function $w$.
\end{lemma}

\begin{proof}    
  We set again, for the sake of brevity,
  \[
  v_h \coloneq w_h - \lproj{h}{k}w.
  \]
Using \eqref{eq:ibp} with $(\tau, v) = (\beta w, v_h)$ to integrate by parts the first term in the right-hand side of \eqref{eq:Errb} along with $\nabla \cdot \beta = 0$, recalling the single-valuedness of $\beta \cdot \normal_F$ and $(\beta \cdot \normal_F) w$ across any interface $F \in \Fhi$, and inserting $w$ into the jump operator after noticing that this quantity is single-valued across interfaces and it vanishes on boundary faces, we arrive at the following decomposition of the error:
  \begin{equation}\label{eq:Errb:basic}
    \begin{aligned}
      \Errb(w;v_h)
      &= -\int_\Omega(w - \lproj{h}{k}w)(\beta\cdot\nabla_h v_h)
      + \int_\Omega\mu(w - \lproj{h}{k}w) v_h
      \\
      &\quad
      + \sum_{F\in\Fh}\int_F(\beta\cdot\normal_F) \avg{w - \lproj{h}{k}w}\jump{v_h}
      + \frac12\sum_{F\in\Fh}\vref{F}\int_F \jump{w - \lproj{h}{k}w} \jump{v_h}.
      \\
      &\eqcolon
      \term_1 + \cdots + \term_4.
    \end{aligned}
  \end{equation}
  We proceed to estimate the terms in the right-hand side.
  \medskip\\
  \underline{\textbf{Estimate of $\term_1$}.}
  For the first term, we use the definition of $\lproj{T}{k}$ along with the fact that $\lproj{T}{0}\beta \cdot \nabla v_T \in \Poly{k-1}(T) \subset \Poly{k}(T)$ to write
  \begin{equation}\label{eq:Errb:T1}
    \begin{aligned}
      \term_1
      &= \sum_{T\in\Th}\int_T(w-\lproj{T}{k}w)[(\beta-\lproj{T}{0}\beta)\cdot\nabla v_T]
      \\
      &\le\sum_{T\in\Th}
      \norm{L^2(T)}{w-\lproj{T}{k}w}\norm{L^\infty(T)^d}{\beta-\lproj{T}{0}\beta}\norm{L^2(T)^d}{\nabla v_T}
      \\
      &\lesssim\sum_{T\in\Th}
      h_T^{r+1}\seminorm{H^{r+1}(T)}{w}~h_T\seminorm{W^{1,\infty}(T)^d}{\beta}~h_T^{-1}\norm{L^2(T)}{v_T}
      \\
      &\le\left(
      \sum_{T\in\Th}\tref^{-2}\underline{\mu}_T^{-1}h_T^{2(r+1)}\seminorm{H^{r+1}(T)}{w}^2
      \right)^{\frac12}\norm{\beta,\mu,h}{v_h}
      \\
      &\le\delta\norm{\beta,\mu,h}{v_h}^2
      + c(\delta)\sum_{T\in\Th}\tref^{-2}\underline{\mu}_T^{-1}h_T^{2(r+1)}\seminorm{H^{r+1}(T)}{w}^2,
    \end{aligned}
  \end{equation}
  where we have used a H\"older inequality with exponents $(2,\infty,2)$ to pass to the second line,
  the approximation properties of the $L^2$-orthogonal projector along with a discrete inverse inequality to pass to the third line,
  a discrete Cauchy--Schwarz inequality on the sum over $T\in\Th$ along with the definition \eqref{eq:tref} of the reference time to pass to the fourth line,
  and a generalized Young inequality to conclude.
  \medskip\\
  \underline{\textbf{Estimate of $\term_2$}.}
  For the second term, a H\"{o}lder inequality with exponents $(\infty,2,2)$
  and the approximation properties of the $L^2$-orthogonal projector readily give
  \begin{equation}\label{eq:Errb:T2}
    \begin{aligned}
      \term_2
      &\lesssim
      \sum_{T\in\Th} \norm{L^\infty(T)}{\mu}^{\frac12} h_T^{r+1}\seminorm{H^{r+1}(T)}{w}~\norm{L^2(T)}{\mu^{\frac12}v_T}
      \\
      \overset{\eqref{eq:tref}}&\le
      \sum_{T\in\Th}\tref^{-\frac12}h_T^{r+1}\seminorm{H^{r+1}(T)}{w}~\norm{L^2(T)}{\mu^{\frac12}v_T}
      \\
      &\le\left(
      \sum_{T\in\Th}\tref^{-2}\underline{\mu}_T^{-1}h_T^{2(r+1)}\seminorm{H^{r+1}(T)}{w}^2
      \right)^{\frac12}\norm{\beta,\mu,h}{v_h}
      \\
      &\le\delta\norm{\beta,\mu,h}{v_h}^2
      + c(\delta)\sum_{T\in\Th}\tref^{-2}\underline{\mu}_T^{-1}h_T^{2(r+1)}\seminorm{H^{r+1}(T)}{w}^2,
    \end{aligned}
  \end{equation}
  where we have used a discrete Cauchy--Schwarz inequality on the sum over $T\in\Th$, noticed that $\tref^{-1}\le\tref^{-2}\underline{\mu}_T^{-1}$, recalled the definition \eqref{eq:norm.beta.mu.h} of the advection-reaction norm in the third inequality, and used a generalized Young inequality to conclude.
  \medskip\\
  \underline{\textbf{Estimate of $\term_3 + \term_4$}.}
  We next write $\term_3 + \term_4 = \sum_{F\in\Fh}\term_{3+4}(F)$ and estimate separately the local contribution on diffusion- and advection-dominated faces.
  
  For all $F\in\Fhd(w)$, we write
  \[
  \begin{aligned}
    |\term_{3+4}(F)|
    &\lesssim\vref{F}\left(
    \norm{L^{p'}(F)}{\avg{w - \lproj{h}{k}w}}
    + \norm{L^{p'}(F)}{\jump{w - \lproj{h}{k}w}}
    \right)~\norm{L^p(F)}{\jump{v_h}}
    \\
    &\lesssim\vref{F} h_F^{r+1-\frac{1}{p'}}\seminorm{W^{r+1,p'}(\TF)}{w}~h_F^{\frac{1}{p'}}
    h_F^{\frac{1-p}{p}}\norm{L^p(F)}{\jump{v_h}}
    \\
    &= \dref{F}(w) \Pe_F(w) h_F^r \seminorm{W^{r+1,p'}(\TF)}{w}~
    h_F^{\frac{1-p}{p}}\norm{L^p(F)}{\jump{v_h}}
    \\
    &\le \dref{F}(w) h_F^r\seminorm{W^{r+1,p'}(\TF)}{w}~
    h_F^{\frac{1-p}{p}}\norm{L^p(F)}{\jump{v_h}},
  \end{aligned}
  \]
where we have used a H\"older inequality with exponents $(\infty,p',p)$ in the first inequality, triangle inequalities followed by the trace approximation properties of the $L^2$-orthogonal projector along with \eqref{eq:p'} to write $1 = h_F^{\frac{1}{p'}}h_F^{\frac{1-p}{p}}$ in the second inequality,
  the definition \eqref{eq:PeF.drefF} of the local Péclet number in the equality,
  and the fact that $\Pe_F(w)\le 1$ to conclude.
  
  For all $F\in\Fha(w)$, on the other hand, the estimate is
  \begin{equation}\label{eq:usedinrem2}
  \begin{aligned}
    |\term_{3+4}(F)|
    &\lesssim\vref{F}^{\frac12}\left(
    \norm{L^2(F)}{\avg{w-\lproj{h}{k}w}}
    + \norm{L^2(F)}{\jump{w-\lproj{h}{k}w}}
    \right)~\vref{F}^{\frac12}\norm{L^2(F)}{\jump{v_h}}
    \\
    &\lesssim\vref{F}^{\frac12}h_F^{r+\frac12}\seminorm{H^{r+1}(\TF)}{w}~\vref{F}^{\frac12}\norm{L^2(F)}{\jump{v_h}},
  \end{aligned}
  \end{equation}
  where we have used a H\"older inequality with exponents $(\infty,2,2)$ in the first inequality and triangle inequalities followed by the approximation properties of the $L^2$-orthogonal projector in the second inequality.
  After applying a discrete H\"older inequality with exponents $(p',p)$ on the sum over diffusive faces,
  a discrete Cauchy--Schwarz inequality on the sum over advective faces,
  and using generalized Young inequalities,
  we arrive at
  \begin{equation}\label{eq:Errb:T3+4}
    \begin{aligned}
      |\term_3 + \term_4|
      &\le
      \delta\left(
      \seminorm{1,p,h}{v_h}^q
      + \norm{\beta,\mu,h}{v_h}^2
      \right)
      + c(\delta) \sum_{F\in\Fha(w)} \vref{F} h_F^{2r+1}\seminorm{H^{r+1}(\TF)}{w}^2
      \\
      &\quad
      + c(\delta)\left(
      \sum_{F\in\Fhd(w)} \dref{F}(w)^{p'} h_F^{rp'}\seminorm{W^{r+1,p'}(\TF)}{w}^{p'}
      \right)^{\frac{q'}{p'}}.      
    \end{aligned}
  \end{equation}
  \\
  \underline{\textbf{Conclusion.}}
  Plugging \eqref{eq:Errb:T1}, \eqref{eq:Errb:T2}, and \eqref{eq:Errb:T3+4} into \eqref{eq:Errb:basic} and recalling that, in each of these estimates, $\delta > 0$ is arbitrary, the conclusion follows.
\end{proof}

\begin{remark}[Comparison with conforming finite elements]\label{newrem-2}
  As already discussed at the end of Section \ref{sec:main}, for $p>2$ the bound \eqref{eq:err.est} compares unfavorably with the conforming FE case in diffusion dominated cases.
  The reason are the terms $\term_2$ for the diffusive part, c.f. \eqref{eq:Erra:estimate:T2}, and $\term_3 + \term_4$ for the advective part, c.f. \eqref{eq:Errb:T3+4}, which behave as $\mathcal{O}(h^{rp'})$ for diffusion dominated faces (instead of $\mathcal{O}(h^{2r})$). One could slightly improve such bounds by the following observations. The polynomial approximation estimate in \eqref{eq:usedinremark} can be pushed further, requiring a higher regularity $\sigma(\nabla w) \in W^{r+1,p'}(\elements{F})$ but yielding a bound of order $h_F^{r+1}$. Furthermore, $\term_3 + \term_4$ could be bounded using advection, as in \eqref{eq:usedinrem2}, also in diffusion dominated cases, thus avoiding the $\mathcal{O}(h^{rp'})$ term. Indeed note that, due to the presence of $\vref{F}$ in $\term_{3+4}(F)$, the bound \eqref{eq:usedinrem2} does not need any assumption on dominant advection. The above modifications would lead to an $\mathcal{O}(h^{(r+1)p'})$ right hand side for $p>2$ in diffusion dominated cases.
\end{remark}

\subsection{Proof of Theorem \ref{thm:convergence}}\label{sec:err.est}

We start by writing
\begin{multline}\label{eq:err.est:basic}
  C_a \norm{1,p,h}{u_h - \lproj{h}{k} u}^q
  + \norm{\beta,\mu,h}{u_h - \lproj{h}{k} u}^2
  \\
  \begin{aligned}[t]
    \overset{\eqref{eq:ah:stability:3},\eqref{eq:bh:coercivity}}&\le
    a_h(u_h, u_h - \lproj{h}{k} u) - a_h(\lproj{h}{k}u, u_h - \lproj{h}{k} u)
    + b_h(u_h - \lproj{h}{k} u, u_h - \lproj{h}{k} u)
    \\
    \overset{\eqref{eq:Erra},\eqref{eq:Errb}}&=
    \Erra(u; u_h - \lproj{h}{k} u)
    + \Errb(u; u_h - \lproj{h}{k} u)
    \eqcolon\Err(u; u_h - \lproj{h}{k} u) \, ,
  \end{aligned}
\end{multline}
where we also used \eqref{eq:strong} and \eqref{eq:discrete} in order to derive the last identity.
For any real number $\delta > 0$, it holds
\[
\begin{aligned}
  \Err(u; u_h - \lproj{h}{k} u)
  \overset{\eqref{eq:Erra:estimate},\eqref{eq:Errb:estimate},\eqref{eq:bh:coercivity}}&\le
  \delta\left(
  a_h(u_h, u_h - \lproj{h}{k} u)
  - a_h(\lproj{h}{k} u, u_h - \lproj{h}{k} u)    
  + b_h(u_h - \lproj{h}{k}u, u_h - \lproj{h}{k}u)
  \right)
  \\
  &\quad
  + \delta\left(
  2\seminorm{1,p,h}{u_h - \lproj{h}{k}u}^q
  + \norm{\beta,\mu,h}{u_h - \lproj{h}{k}u}^2
  \right)
  + c(\delta) E_h^k
  \\
  &\le\delta\Err(u; u_h - \lproj{h}{k}u)
  + \delta\left(
  2\seminorm{1,p,h}{u_h - \lproj{h}{k}u}^q
  + \norm{\beta,\mu,h}{u_h - \lproj{h}{k}u}^2
  \right)
  + c(\delta) E_h^k,
\end{aligned}
\]
where $c(\delta)$ denotes the largest value between \eqref{eq:Erra:estimate} and \eqref{eq:Errb:estimate}, while $E_h^k$ gathers all the terms multiplied by $c(\delta)$ in the sum of the right-hand sides of \eqref{eq:Erra:estimate} and \eqref{eq:Errb:estimate}.
For any $\delta < 1$, this gives
\begin{equation}\label{eq:Err:estimate}
  \Err(u; u_h - \lproj{h}{k} u)
  \le\frac{2\delta}{1-\delta}
  \seminorm{1,p,h}{u_h - \lproj{h}{k}u}^q
  + \frac{\delta}{1-\delta} \norm{\beta,\mu,h}{u_h - \lproj{h}{k}u}^2
  + \frac{c(\delta)}{1-\delta} E_h^k.
\end{equation}
Let now $\epsilon$ and $\delta_\epsilon$ denote two real numbers such that $0 < \epsilon < \frac{2}{2+C_a}$ and $0 < \delta_\epsilon < \frac12\min\left(1, \epsilon C_a\right)$.
Plugging \eqref{eq:Err:estimate} with $\delta = \delta_\epsilon$ into \eqref{eq:err.est:basic}, noticing that, by definition, $\frac{2\delta_\epsilon}{1-\delta_\epsilon} < \frac{\epsilon}{1-\delta_\epsilon}C_a$, rearranging, and multiplying the resulting inequality by $(1-\delta_\epsilon)$, we get
\[
(1 - \delta_\epsilon - \epsilon)C_a\norm{1,p,h}{u_h - \lproj{h}{k}u}^q
+ (1 - 2 \delta_\epsilon) \norm{\beta,\mu,h}{u_h - \lproj{h}{k} u}^2
\le c(\delta_\epsilon) E_h^k.
\]
We conclude noticing that, by definition of $\epsilon$ and $\delta_\epsilon$, $1 - \delta_\epsilon - \epsilon > 1 - \frac{\epsilon}{2} C_a - \epsilon > 0$, and $1 - 2\delta_\epsilon > 0$.


\section{Numerical tests}\label{sec:num}

In this section, we investigate from the practical standpoint the error estimates derived in Theorem \ref{thm:convergence} through some numerical experiments.
The computational domain for all the tests developed in this section is the standard unit square $\Omega = (0,1)^2$. 
In order to analyze the numerical convergence rate, we consider a family of five triangular meshes $\Th$ with decreasing diameters, namely 
\[
h \in \{0.4714,\ 0.2215,\ 0.1189,\ 0.0588,\ 0.0314 \}.
\]
Starting from the coarsest mesh, the subsequent meshes are obtained by (approximately) halving the meshsize. 
Due to the nonlinearity of the problem for $p \neq 2$, we use a fixed-point strategy to compute the discrete solution. We set the maximum number of iterations to $N_{max}=500$, the tolerance for the relative residual error to $\varepsilon = 10^{-10}$, and we take the initial guess as the discrete solution of the problem with $p=2$.

\subsection{Example 1} \label{sec:ex1}
In the present example we consider problem \eqref{eq:strong} with the following exact solution, velocity field and reaction terms
\[
u(x,y) \coloneqq \sin (x+0.1) \cos( y+0.1),\quad 
\beta (x,y) \coloneqq \begin{bmatrix} \sin(x)\cos(y) \\ -\sin(y)\cos(x) \end{bmatrix},\quad \mu(x,y) \coloneqq 1 \, .
\]
The problem is investigated for different values of the Sobolev index $p$, specifically for the following choices
\[
p \in \{1.5 , 1.75, 2, 2.5, 3\}.
\]
The source term $f$ and the nonhomogeneous Dirichlet boundary condition are taken in accordance with $p$, the above analytical solution, and the remaining terms in the equation. 

Furthermore, we introduce a coefficient $\nu$ which multiplies the diffusive term $-\nabla\cdot \sigma (\nabla u)$ and allows to control the relative magnitude of the diffusion and advection terms.
Specifically, we set $\nu = 1$ for a diffusion-dominated regime and $\nu = 10^{-4}$ for an advection-dominated regime.

We compute the discrete solution $u_h\in\Poly{k}(\Th)$ for $k \in \{1,2,3\}$, in both the diffusion-dominated and advection-dominated regimes, with the aim of analyzing the numerical behavior of the error quantity 
\[
\mathrm{ERR}_h \coloneqq \left( \nu \norm{1,p,h}{u-u_h}^q + \norm{\beta,\mu,h}{u - u_h}^2\right)^{\frac{1}{2}}
\]
with $q$ defined by \eqref{eq:q}. \\
In Figure \ref{fig:ex1 D dominated} and Figure \ref{fig:ex1 A dominated} we show, respectively, convergence graphs for the diffusion-dominated and advection-dominated case. 
The numbers appearing in the yellow boxes, directly on the graph segments in our plots, represent the reduction rate associated to two subsequent errors, that is
$$
m_{h_1,h_2} = \frac{\log(\mathrm{ERR}_{h_2}-\mathrm{ERR}_{h_1})}{\log(h_2 - h_1)} 
$$
where $h_1,h_2$ here denote the two mesh sizes associated to the segment endpoints.

In the first setting, the results are in agreement with the theoretical estimates, but exhibit a higher error reduction rate with respect to the theoretical prediction. Indeed, for $p < 2$ we observe a reduction of the error behaving as $\mathcal{O}(h^k)$ instead of $\mathcal{O}(h^{\frac{kp}{2}})$ while, for $p>2$, the error decreases at a rate of $\mathcal{O}(h^{\frac{kp}{2}})$ instead of $\mathcal{O}(h^{\frac{kp'}{2}})$. In both cases, the reduction rate corresponds to that obtained by the best approximant to the solution $u$ in $\Poly{k}(\Th)$; we better investigate this aspect in the next example.

In the advection-dominated regime, on the other hand, the observed convergence rates closely match the theoretical estimates, i.e. $\mathrm{ERR}_h$ exhibits an $\mathcal{O}(h^{k+\frac{1}{2}})$ decay. In particular, we can observe the additional $h^{\frac12}$ factor which is gained due to the convection robustness of the method.

\begin{figure} [t!]
  \begin{subfigure}{.3\linewidth}
    \centering
    \includegraphics[width=1\linewidth]{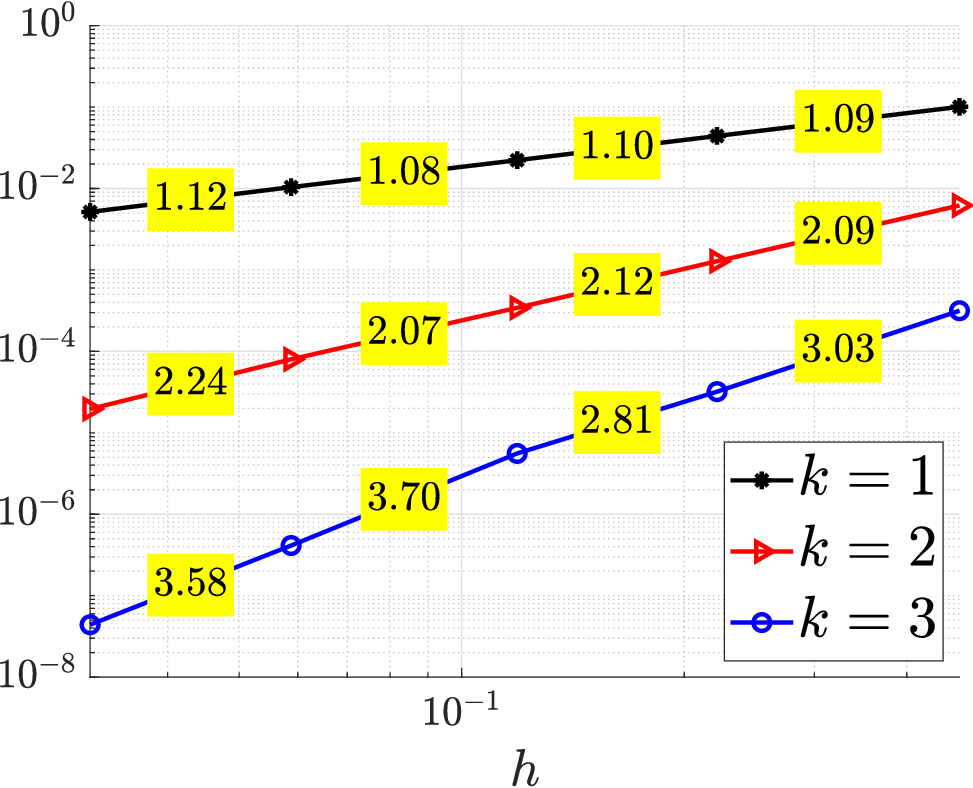}
    \caption{$p=1.5$}
  \end{subfigure}%
  \hspace{0.05\linewidth}%
  \begin{subfigure}{.3\linewidth}
    \centering
    \includegraphics[width=\linewidth]{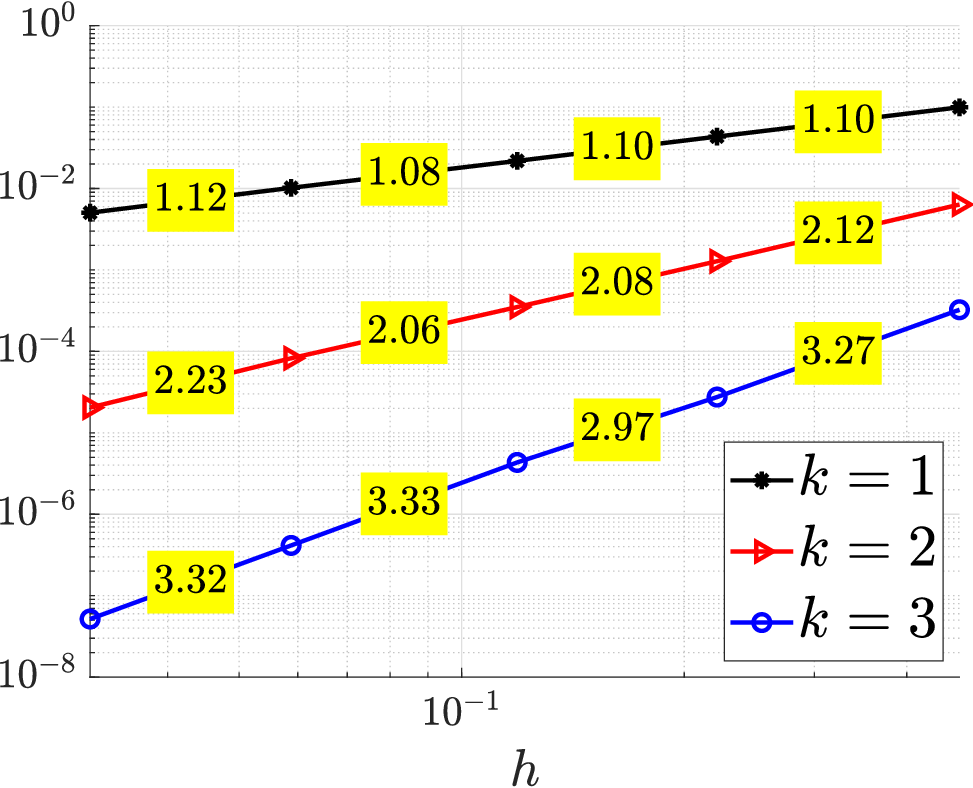}
    \caption{$p=1.75$}
  \end{subfigure}%
  \hspace{0.05\linewidth}%
  \begin{subfigure}{.3\linewidth}
    \centering
    \includegraphics[width=\linewidth]{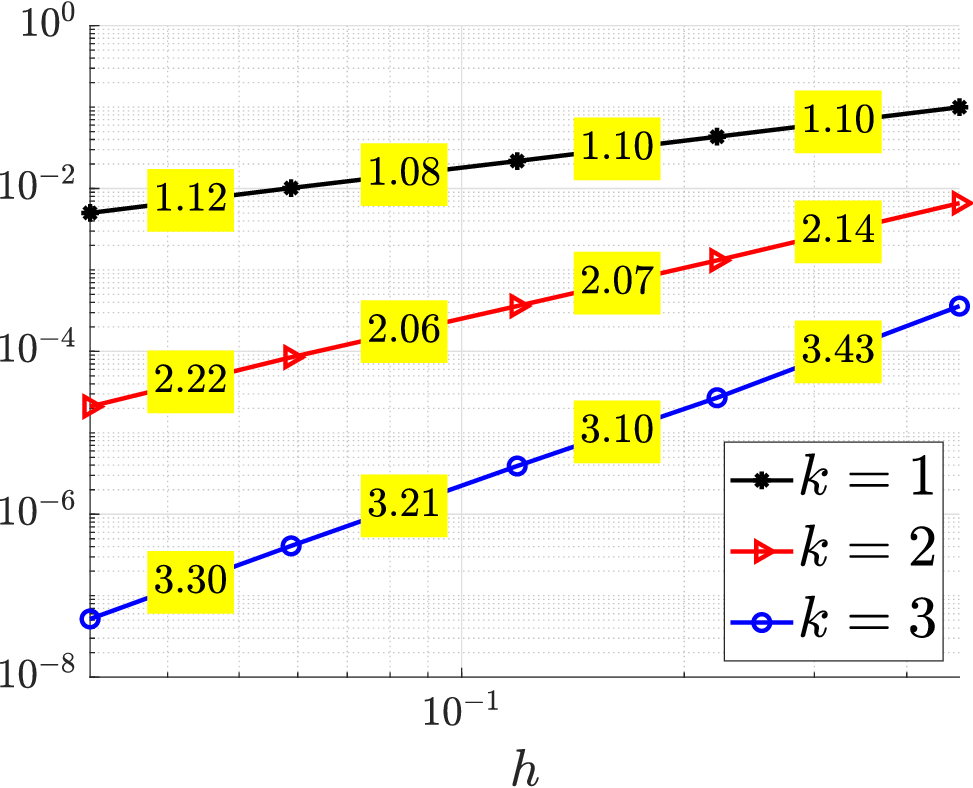}
    \caption{$p=2$}
  \end{subfigure}
\begin{center}
  \begin{subfigure}{.3\linewidth}
    \centering
    \includegraphics[width=\linewidth]{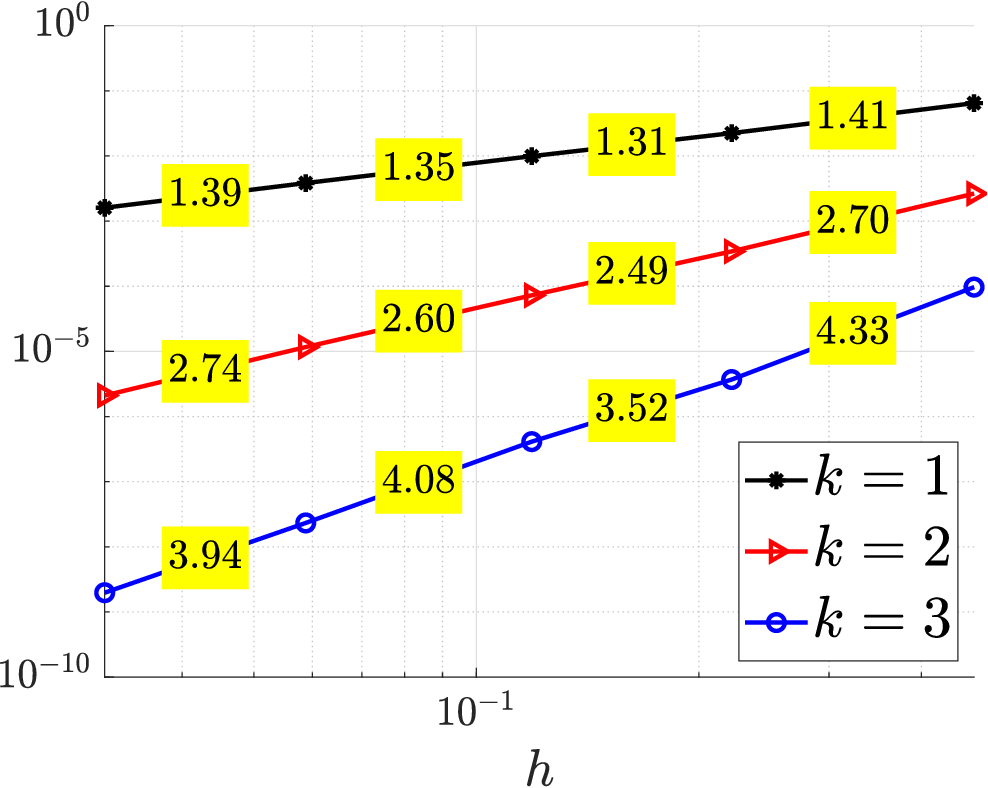}
    \caption{$p=2.5$}
  \end{subfigure}%
  \hspace{0.05\linewidth}%
  \begin{subfigure}{.3\linewidth}
    \centering
    \includegraphics[width=\linewidth]{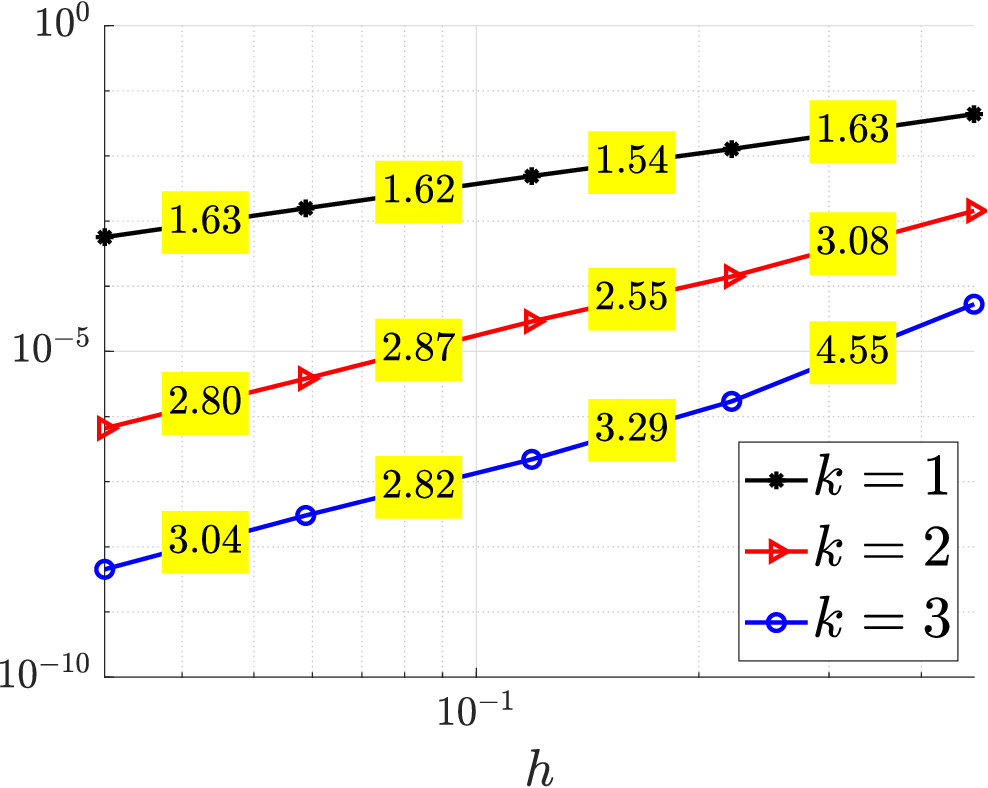}
    \caption{$p=3$}
  \end{subfigure}
  \end{center}
  \caption{Example of Section \ref{sec:ex1}. Convergence rate of $\mathrm{ERR}_h$ in the diffusion-dominated regime. Theoretical convergence rate: $\frac{kp}{2}$ for $p\leq 2$ and $\frac{kp'}{2}$ for $p>2$.}
  \label{fig:ex1 D dominated}
\end{figure}

\begin{figure} [t!]
  \begin{subfigure}{.3\linewidth}
    \centering
    \includegraphics[width=1\linewidth]{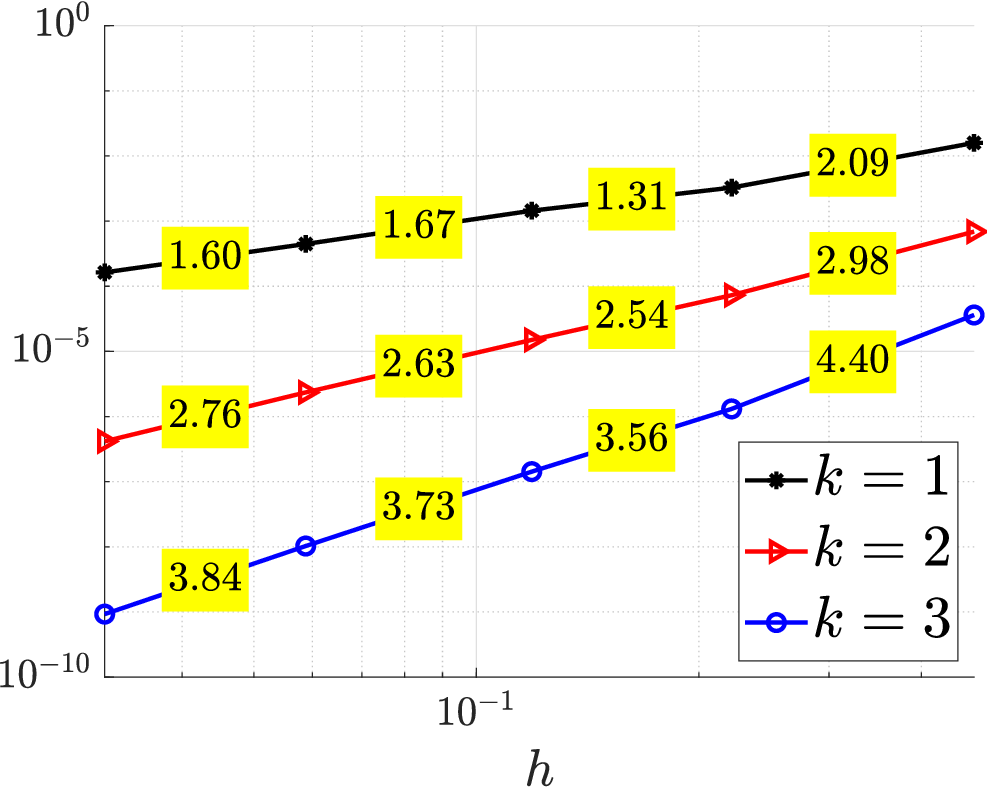}
    \caption{$p=1.5$}
  \end{subfigure}%
  \hspace{0.05\linewidth}%
  \begin{subfigure}{.3\linewidth}
    \centering
    \includegraphics[width=\linewidth]{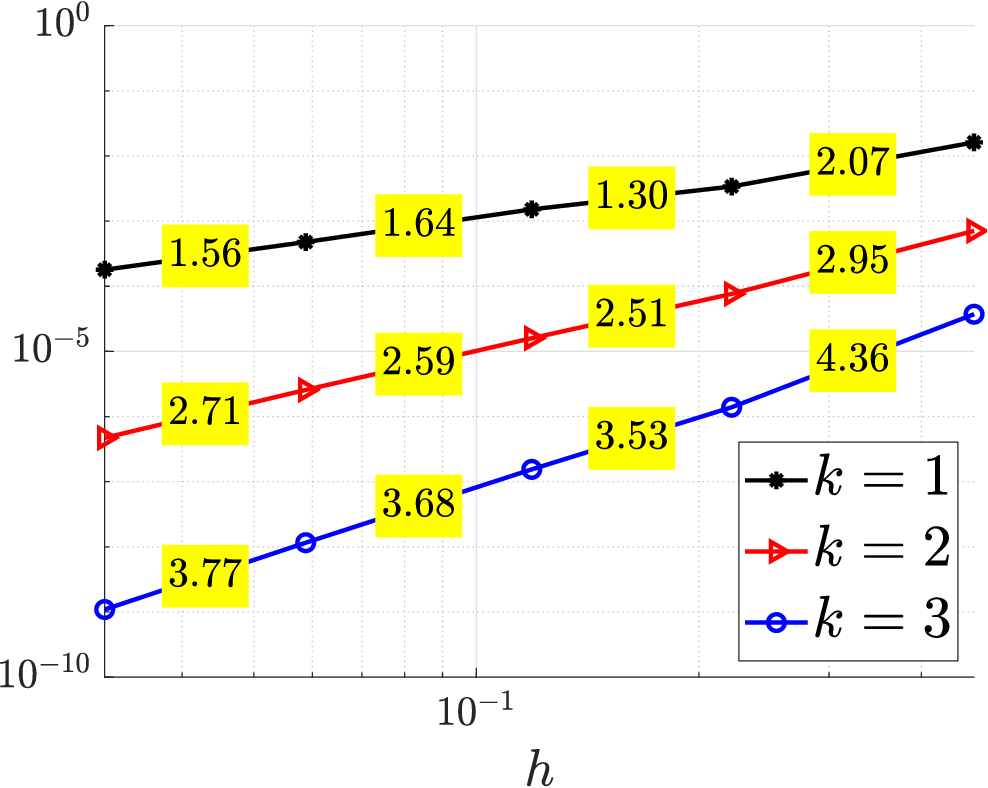}
    \caption{$p=1.75$}
  \end{subfigure}%
  \hspace{0.05\linewidth}%
  \begin{subfigure}{.3\linewidth}
    \centering
    \includegraphics[width=\linewidth]{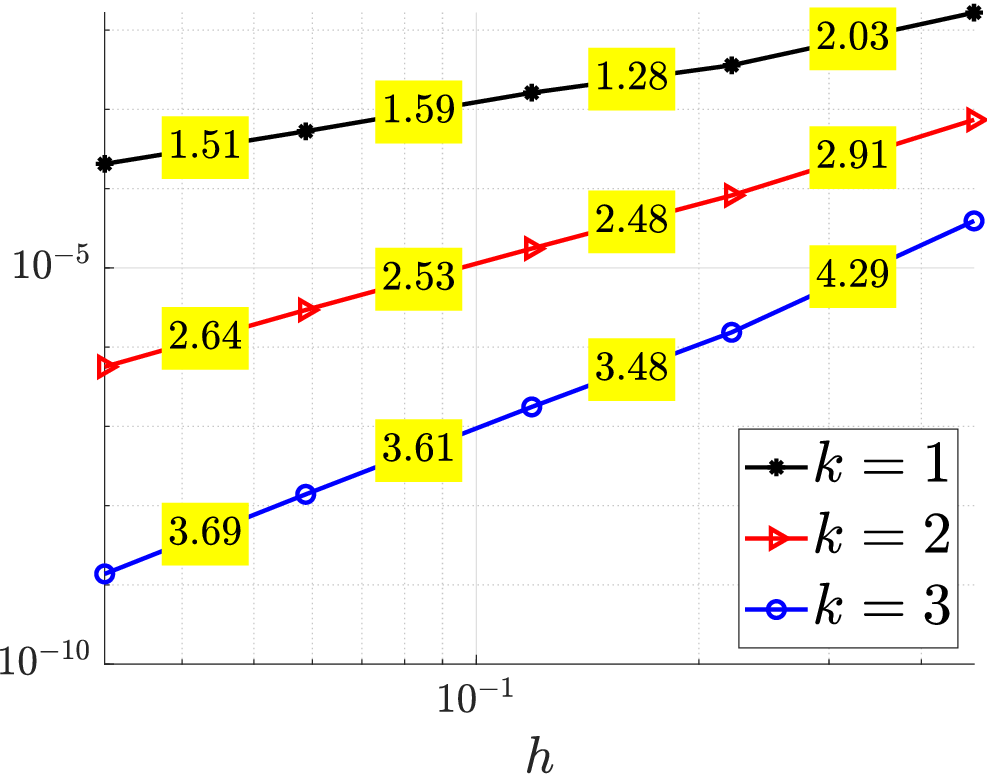}
    \caption{$p=2$}
  \end{subfigure}
\begin{center}
  \begin{subfigure}{.3\linewidth}
    \centering
    \includegraphics[width=\linewidth]{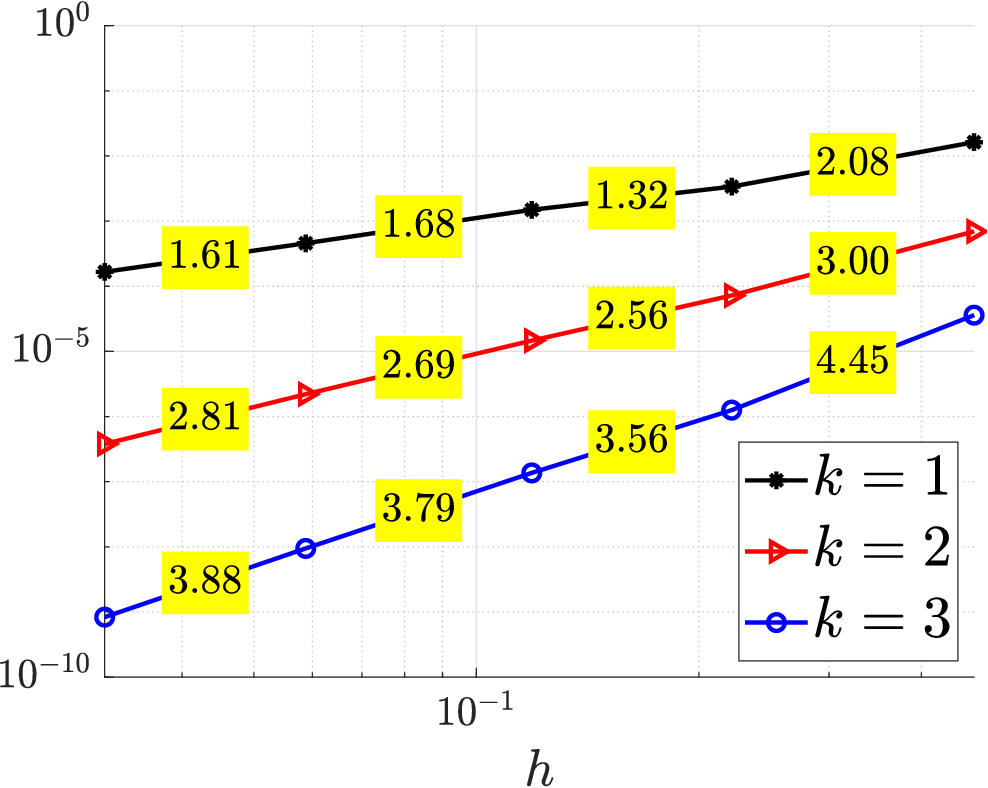}
    \caption{$p=2.5$}
  \end{subfigure}%
  \hspace{0.05\linewidth}%
  \begin{subfigure}{.3\linewidth}
    \centering
    \includegraphics[width=\linewidth]{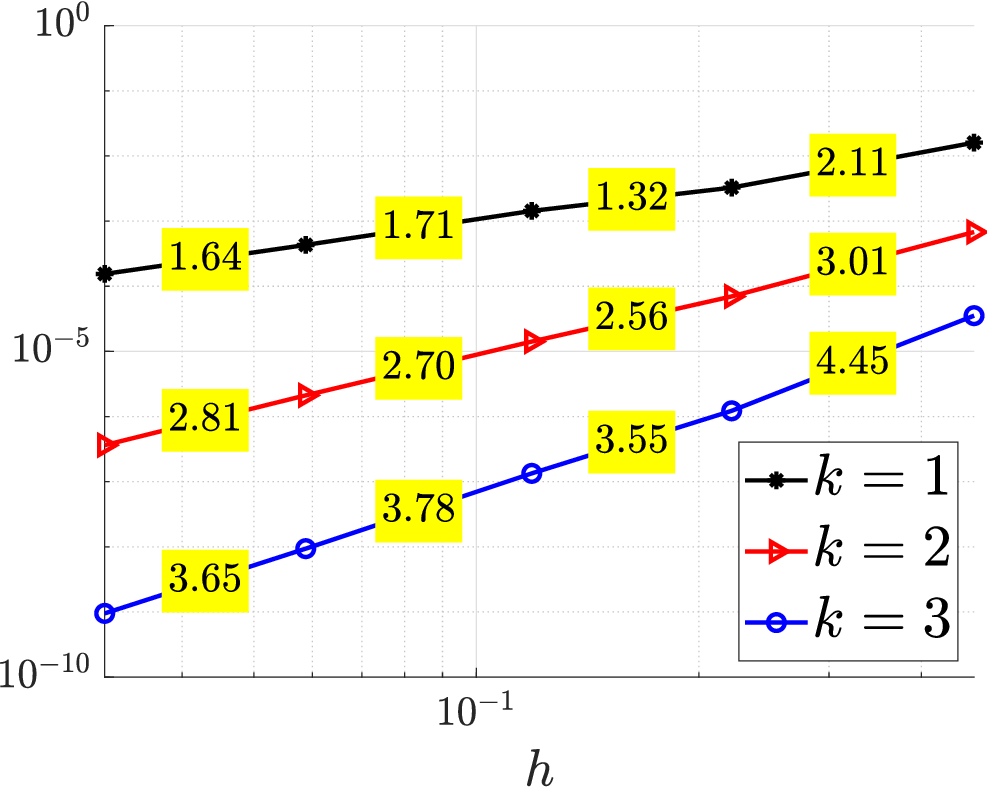}
    \caption{$p=3$}
  \end{subfigure}
  \end{center}
  \caption{Example of Section \ref{sec:ex1}. Convergence rate of $\mathrm{ERR}_h$ in the advection-dominated regime. Theoretical convergence rate: $k+\frac{1}{2}$ for all $p$.}
  \label{fig:ex1 A dominated}
\end{figure}

\subsection{Example 2} \label{sec:ex2}
In the second example, we consider problem \eqref{eq:strong} without the presence of advection and reaction phenomena.
The motivation of this second example is to better investigate the ``higher than expected'' reduction rate for the diffusion dominated case in Example 1.
We therefore directly set $\nu=1$, $\beta=0$, $\mu=0$ (pure diffusion) and choose
the right-hand side and the Dirichlet boundary condition in accordance with two distinct solutions (the exponential $(p,k)$-dependent solution was originally proposed in \cite{Di-Pietro.Droniou.ea:21}) :
\begin{enumerate}
\item[•] $u(x,y) \coloneqq \frac{1}{10}\mathrm{exp}\left[-10\left(\left|-x+0.5\right|^{p+\frac{k+2}{4}} + \left|-y+0.5\right|^{p+\frac{k+2}{4}}\right)\right]$;
\item[•] $u(x,y) \coloneqq \left(x-\frac{1}{2}\right)^2\left(y-\frac{1}{2}\right)^2$
\end{enumerate}

Here, the difference with respect to the preceding example is that the gradient $\nabla u$ vanishes at the point of coordinates $(0.5,0.5)$ for the exponential solution, and in the region $\left\{(x,y) \in \Omega \st x=\frac{1}{2} \text{ or } y=\frac{1}{2} \right\}$ for the polynomial solution, while, in the previous case, the solution had a non-zero gradient over the entire domain (which may determine a favorable situation for $p<2$, see for instance \cite{Di-Pietro.Droniou.ea:21}). In this respect, the the polynomial solution can be more challenging than the exponential one, as will be confirmed by the following results. 
Furthermore, the coarsest mesh, with meshsize $h=0.4714$, is removed and replaced by two new meshes obtained by halving subsequently the finest mesh: this results in two new meshsizes with $h=0.0158$ and $h=0.0082$.
For both solutions we have checked, by direct computation and/or numerically, that the flux $\sigma$ is sufficiently regular for the estimates of Theorem \ref{thm:convergence} to hold.

As in the previous example, we compute the error term $\mathrm{ERR}_h$ (in this case with $\nu=1$, $\beta=0$, $\mu=0$) considering $u_h \in\Poly{k}(\Th)$ for all combinations $(p,k)$ with $p \in \{1.5 , 1.75 \}$ and $k \in \{1, 2\}$.  
The outcome in Figure \ref{fig:ex 2 exp}, where $\mathrm{ERR}_h$ is plotted for the exponential solution, is similar to the previous example despite the different solution (now with vanishing gradient in a point of the domain) and the finer meshes adopted. Our current conclusions are that, probably, such behaviour is still pre-asymptotic, as is the case for the HHO method on meshes of similar size (cf., in particular, \cite[Table~4]{Di-Pietro.Droniou.ea:21}).

On the other hand, the results in Figure \ref{fig:ex 2 poly}, showing the convergence rates for the polynomial solution, are aligned with the expected convergence rate on the light of Theorem \ref{thm:convergence}. Indeed, an $\mathcal{O}(h^{\frac{kp}{2}})$ decay of $\mathrm{ERR}_h$ can be observed (especially for the finer meshes), which confirms from the practical side the sharpness of the theoretical results. 

\begin{figure} [t!]
\begin{center}
  \begin{subfigure}{.3\linewidth}
    \centering
    \includegraphics[width=\linewidth]{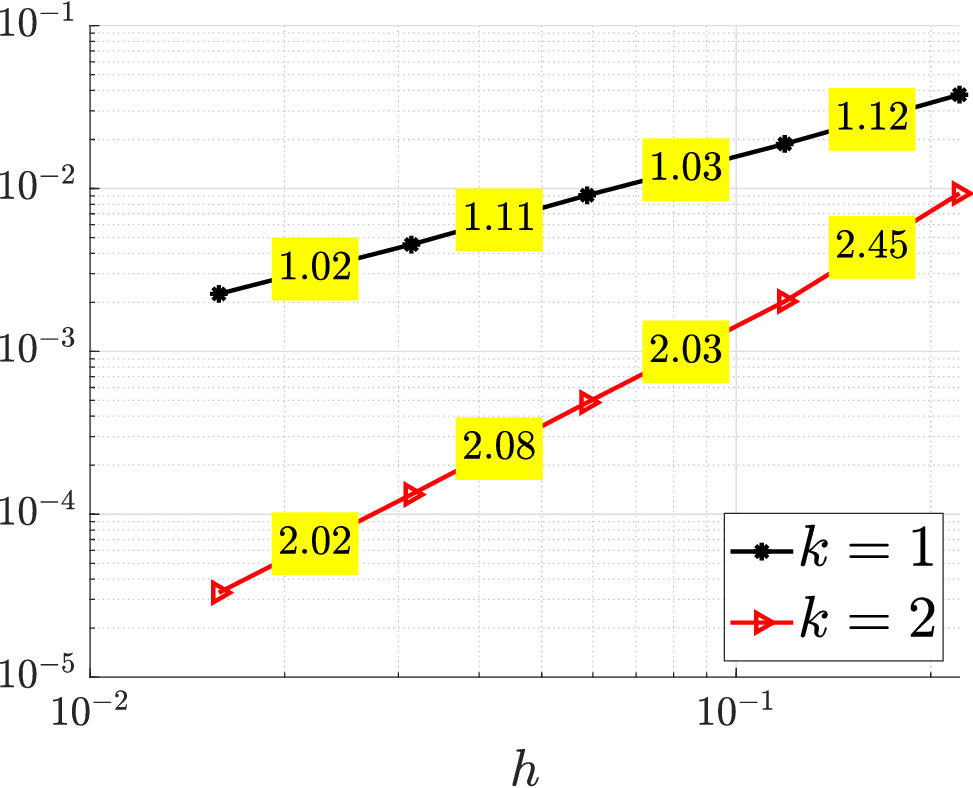}
    \caption{$p=1.5$}
  \end{subfigure}%
      \hspace{0.05\linewidth}%
  \begin{subfigure}{.3\linewidth}
    \centering
    \includegraphics[width=\linewidth]{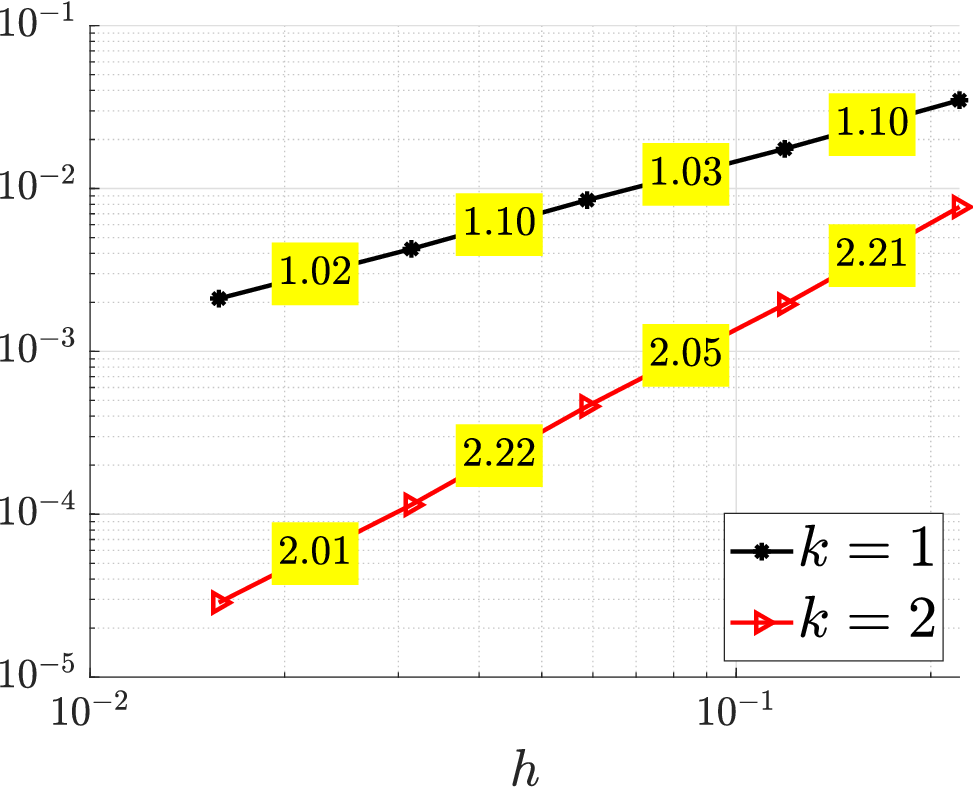}
    \caption{$p=1.75$}
  \end{subfigure}
  \end{center}
  \caption{Example of Section \ref{sec:ex2} with exponential solution. Convergence rate of $\mathrm{ERR}_h$. Theoretical convergence rate: $\frac{kp}{2}$ for $p \in \{1.5, 1.75 \}$.}
  \label{fig:ex 2 exp}
\end{figure}

\begin{figure} [t!]
\begin{center}
  \begin{subfigure}{.3\linewidth}
    \centering
    \includegraphics[width=\linewidth]{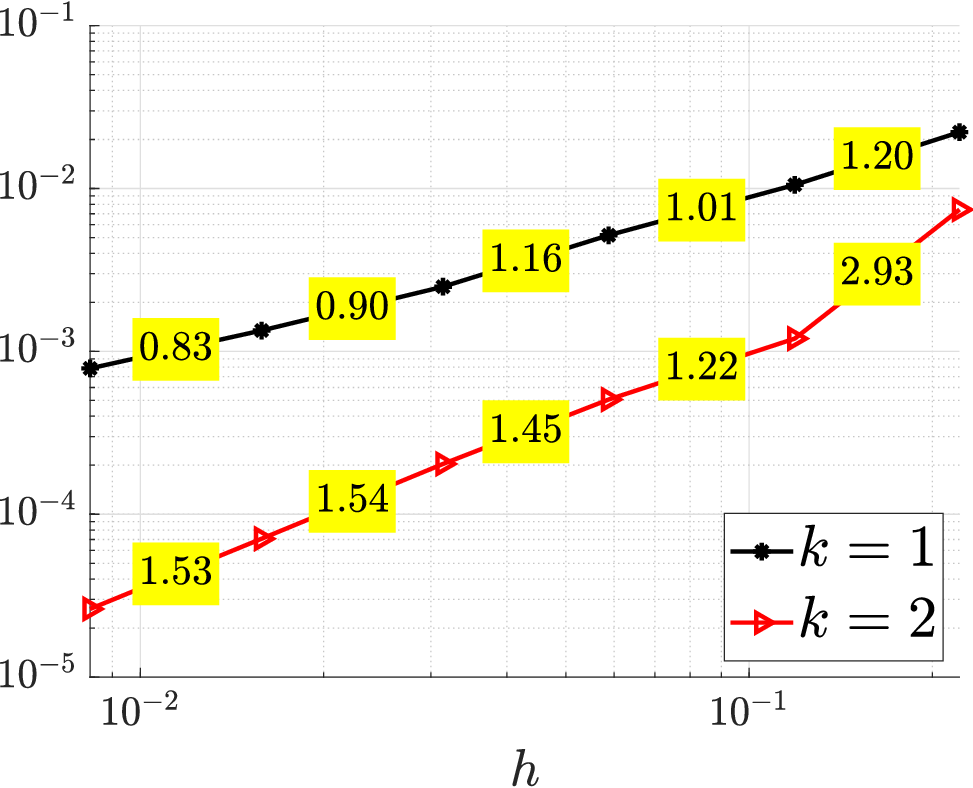}
    \caption{$p=1.5$}
  \end{subfigure}%
      \hspace{0.05\linewidth}%
  \begin{subfigure}{.3\linewidth}
    \centering
    \includegraphics[width=\linewidth]{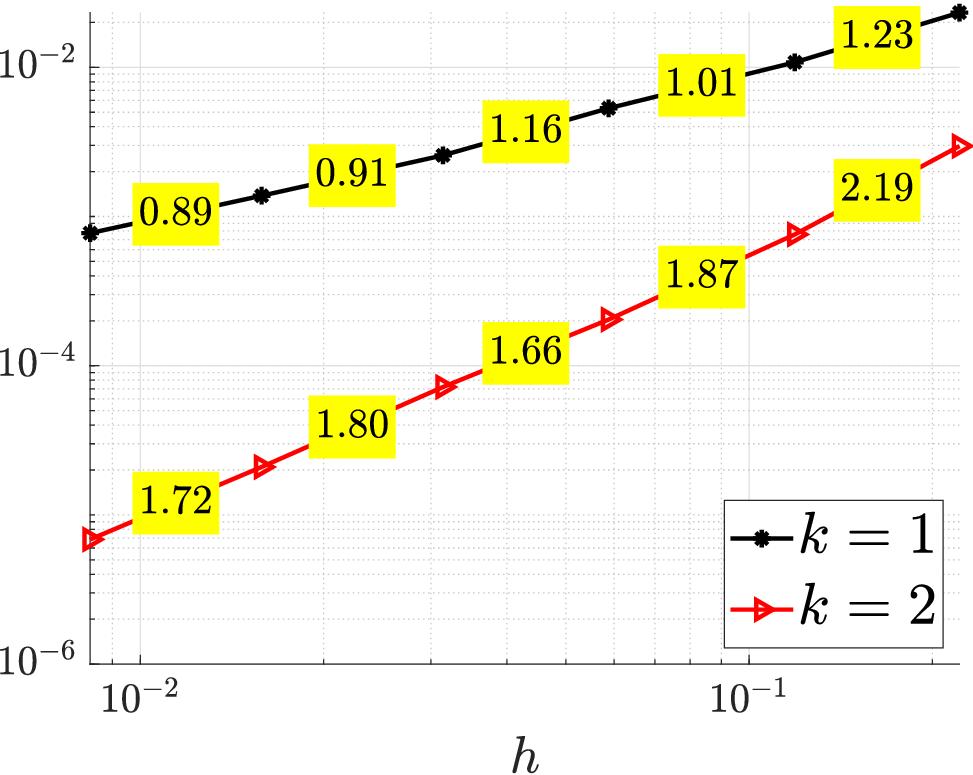}
    \caption{$p=1.75$}
  \end{subfigure}
  \end{center}
  \caption{Example of Section \ref{sec:ex2} with polynomial solution. Convergence rate of $\mathrm{ERR}_h$. Theoretical convergence rate: $\frac{kp}{2}$ for $p \in \{1.5, 1.75 \}$.}
  \label{fig:ex 2 poly}
\end{figure}


\section*{Acknowledgements}

The present results where partially supported by the European Union (ERC Synergy, NEMESIS, project number 101115663).
Views and opinions expressed are however those of the author(s) only and do not necessarily reflect those of the European Union or the European Research Council Executive Agency. 


\printbibliography

@Article{         Antonietti.Giani.ea:13,
  Author        = {Antonietti, Paola F. and Giani, Stefano and Houston,
                  Paul},
  Title         = {{$hp$}-version composite discontinuous {G}alerkin methods
                  for elliptic problems on complicated domains},
  Journal       = {SIAM J. Sci. Comput.},
  Volume        = {35},
  Year          = {2013},
  Number        = {3},
  Pages         = {A1417--A1439},
  DOI           = {10.1137/120877246}
}

@Article{         Arnold.Brezzi.ea:01,
  Author        = {Arnold, Douglas N. and Brezzi, Franco and Cockburn,
                  Bernardo and Marini, L. Donatella},
  Title         = {Unified analysis of discontinuous {G}alerkin methods for
                  elliptic problems},
  Journal       = {SIAM J. Numer. Anal.},
  Volume        = {39},
  Year          = {2001},
  Number        = {5},
  Pages         = {1749--1779},
  DOI           = {10.1137/S0036142901384162}
}

@Article{         Ayuso.Marini:09,
  Author        = {Ayuso, Blanca and Marini, L. Donatella},
  Title         = {Discontinuous {G}alerkin methods for
                  advection-diffusion-reaction problems},
  Journal       = {SIAM J. Numer. Anal.},
  Volume        = {47},
  Year          = {2009},
  Number        = {2},
  Pages         = {1391--1420},
  DOI           = {10.1137/080719583}
}

@Article{         Baker:77,
  Author        = {Baker, G.~A.},
  Title         = {Finite element methods for elliptic equations using
                  nonconforming elements},
  Journal       = {Math. Comp.},
  Year          = {1977},
  Volume        = {31},
  Pages         = {45--49},
  Number        = {137},
  DOI           = {10.2307/2005779}
}

@Article{         Barrett.Liu:93,
  Author        = {Barrett, J.W. and Liu, W.B.},
  Title         = {Finite element approximation of the p-Laplacian},
  Journal       = {Math. Comp.},
  Year          = {1993},
  Volume        = {61},
  Pages         = {523--537},
  DOI           = {10.2307/2153239}
}

@Article{         Bassi.Botti.ea:12,
  Author        = {Bassi, F. and Botti, L. and Colombo, A. and Di Pietro, D.
                  A. and Tesini, P.},
  Title         = {On the flexibility of agglomeration based physical space
                  discontinuous {G}alerkin discretizations},
  Journal       = {J. Comput. Phys.},
  Volume        = {231},
  Year          = {2012},
  Number        = {1},
  Pages         = {45--65},
  DOI           = {10.1016/j.jcp.2011.08.018}
}

@Article{         Bassi.Botti.ea:14,
  Author        = {Bassi, Francesco and Botti, Lorenzo and Colombo,
                  Alessandro},
  Title         = {Agglomeration-based physical frame d{G} discretizations:
                  an attempt to be mesh free},
  Journal       = {Math. Models Methods Appl. Sci.},
  Volume        = {24},
  Year          = {2014},
  Number        = {8},
  Pages         = {1495--1539},
  DOI           = {10.1142/S0218202514400028}
}

@Article{         Beirao-da-Veiga.Dassi.ea:21,
  Author        = {Beir\~{a}o da Veiga, L. and Dassi, F. and Vacca, G.},
  Title         = {Vorticity-stabilized virtual elements for the Oseen
                  equation},
  Journal       = {Math. Models Methods Appl. Sci.},
  Volume        = {31},
  Number        = {14},
  Pages         = {3009-3052},
  Year          = {2021},
  DOI           = {10.1142/S0218202521500688}
}

@Article{         Beirao-da-Veiga.Dassi.ea:21*1,
  Author        = {Beirao da Veiga, Lourenco and Dassi, Franco and Lovadina,
                  Carlo and Vacca, Giuseppe},
  Title         = {SUPG-stabilized virtual elements for diffusion-convection
                  problems: a robustness analysis},
  Journal       = {ESAIM: M2AN},
  Year          = {2021},
  Volume        = {55},
  Number        = {5},
  Pages         = {2233-2258},
  DOI           = {10.1051/m2an/2021050}
}

@Article{         Beirao-da-Veiga.Dassi.ea:23,
  Author        = {Beir\~{a}o da Veiga, L. and Dassi, F. and Vacca, G.},
  Title         = {{Pressure robust SUPG-stabilized finite elements for the
                  unsteady Navier–Stokes equation}},
  Journal       = {IMA J. Numer. Anal.},
  Year          = {2023},
  DOI           = {10.1093/imanum/drad021}
}

@Article{         Brezzi.Marini.ea:04,
  Author        = {Brezzi, F. and Marini, L. D. and S\"{u}li, E.},
  Title         = {Discontinuous {G}alerkin methods for first-order
                  hyperbolic problems},
  Journal       = {Math. Models Methods Appl. Sci.},
  Volume        = {14},
  Year          = {2004},
  Number        = {12},
  Pages         = {1893--1903},
  DOI           = {10.1142/S0218202504003866}
}

@Article{         Burman.Ern:08,
  Author        = {Burman, E. and Ern, A.},
  Title         = {Discontinuous {G}alerkin approximation with discrete
                  variational principle for the nonlinear {L}aplacian},
  Journal       = {C. R. Math. Acad. Sci. Paris},
  Volume        = {346},
  Year          = {2008},
  Number        = {17-18},
  Pages         = {1013--1016},
  DOI           = {10.1016/j.crma.2008.07.005}
}

@Article{         Castillo.Cockburn.ea:00,
  Author        = {Castillo, Paul and Cockburn, Bernardo and Perugia, Ilaria
                  and Sch\"{o}tzau, Dominik},
  Title         = {An A Priori Error Analysis of the Local Discontinuous
                  Galerkin Method for Elliptic Problems},
  Journal       = {SIAM J. Numer. Anal.},
  Volume        = {38},
  Number        = {5},
  Pages         = {1676-1706},
  Year          = {2000},
  DOI           = {10.1137/S0036142900371003}
}

@Article{         Castillo.Cockburn.ea:02,
  Author        = {Castillo, P. and Cockburn, B. and Perugia, I. and
                  Schötzau, D.},
  Title         = {Local discontinuous Galerkin methods for elliptic
                  problems},
  Journal       = {Comm. Numer. Methods Engrg.},
  Volume        = {18},
  Number        = {1},
  Pages         = {69-75},
  DOI           = {10.1002/cnm.471},
  Year          = {2002}
}

@Article{         Cockburn.Shu:91,
  Author        = {Cockburn, B. and Shu, C.-W.},
  Title         = {The {Runge-Kutta} local projection {$P\sp
                  1$}-discontinuous-{G}alerkin finite element method for
                  scalar conservation laws},
  Journal       = {RAIRO Mod\'el. Math. Anal. Num\'er.},
  Year          = {1991},
  Volume        = {25},
  Pages         = {337--361},
  Number        = {3},
  DOI           = {10.1051/m2an/1991250303371}
}

@Article{         Del-Pezzo.Lombardi.ea:12,
  Author        = {Del Pezzo, Leandro M. and Lombardi, Ariel L. and
                  Mart\'{i}nez, Sandra},
  Title         = {Interior penalty discontinuous {G}alerkin {FEM} for the
                  {$p(x)$}-{L}aplacian},
  Journal       = {SIAM J. Numer. Anal.},
  Volume        = {50},
  Year          = {2012},
  Number        = {5},
  Pages         = {2497--2521},
  DOI           = {10.1137/110820324}
}

@Article{         Di-Pietro.Droniou.ea:21,
  Author        = {Di Pietro, D. A. and Droniou, J. and Harnist, A.},
  Title         = {Improved error estimates for {Hybrid High-Order}
                  discretizations of {Leray--Lions} problems},
  Year          = {2021},
  Journal       = {Calcolo},
  Volume        = {58},
  Number        = {19},
  DOI           = {10.1007/s10092-021-00410-z}
}

@Book{            Di-Pietro.Droniou:20,
  Author        = {Di Pietro, D. A. and Droniou, J.},
  Title         = {The {Hybrid High-Order} method for polytopal meshes},
  Subtitle      = {Design, analysis, and applications},
  Publisher     = {Springer International Publishing},
  Year          = {2020},
  Series        = {Modeling, Simulation and Application},
  Volume        = {19},
  DOI           = {10.1007/978-3-030-37203-3}
}

@Article{         Di-Pietro.Droniou:23,
  Title         = {A polytopal method for the {Brinkman} problem robust in
                  all regimes},
  Author        = {Di Pietro, D. A. and Droniou, J.},
  Journal       = {Comput. Meth. Appl. Mech. Engrg.},
  Year          = {2023},
  Volume        = {409},
  Number        = {115981},
  DOI           = {10.1016/j.cma.2023.115981}
}

@Article{         Di-Pietro.Ern.ea:08,
  Author        = {Di Pietro, D. A. and Ern, A. and Guermond, J.-L.},
  Title         = {Discontinuous {G}alerkin methods for anisotropic
                  semi-definite diffusion with advection},
  Journal       = {SIAM J. Numer. Anal.},
  Volume        = {46},
  Number        = {2},
  Pages         = {805--831},
  Year          = {2008},
  DOI           = {10.1137/060676106}
}

@Article{         Di-Pietro.Ern:10,
  Author        = {Di Pietro, D. A. and Ern, A.},
  Title         = {Discrete functional analysis tools for discontinuous
                  {G}alerkin methods with application to the incompressible
                  {N}avier--{S}tokes equations},
  Journal       = {Math. Comp.},
  Volume        = {79},
  Year          = {2010},
  Pages         = {1303--1330},
  DOI           = {10.1090/S0025-5718-10-02333-1}
}

@Book{            Di-Pietro.Ern:12,
  Author        = {Di Pietro, D. A. and Ern, A.},
  Title         = {Mathematical aspects of discontinuous {G}alerkin methods},
  Series        = {Math\'ematiques \& Applications (Berlin) [Mathematics \&
                  Applications]},
  Volume        = {69},
  Publisher     = {Springer, Heidelberg},
  Year          = {2012},
  DOI           = {10.1007/978-3-642-22980-0}
}

@Article{         Diening.Ettwein:08,
  Author        = {Diening, Lars and Ettwein, Frank},
  Title         = {Fractional estimates for non-differentiable elliptic
                  systems with general growth},
  Journal       = {Forum Math.},
  Volume        = {20},
  Year          = {2008},
  Number        = {3},
  Pages         = {523--556},
  DOI           = {10.1515/FORUM.2008.027}
}

@Article{         Han.Hou:21,
  Author        = {Han, Yongbin and Hou, Yanren},
  Title         = {{Semirobust analysis of an H(div)-conforming DG method
                  with semi-implicit time-marching for the evolutionary
                  incompressible Navier–Stokes equations}},
  Journal       = {IMA J. Numer. Anal.},
  Volume        = {42},
  Number        = {2},
  Pages         = {1568-1597},
  Year          = {2021},
  DOI           = {10.1093/imanum/draa104}
}

@Article{         Hirn:13,
  Author        = {Hirn, A.},
  Title         = {Approximation of the p-{S}tokes Equations with Equal-Order
                  Finite Elements},
  Journal       = {J. Math. Fluid Mech.},
  Volume        = {15},
  Year          = {2013},
  Pages         = {65--88},
  DOI           = {10.1007/s00021-012-0095-0}
}

@Article{         Leray.Lions:65,
  Author        = {Leray, J. and Lions, J.-L.},
  Title         = {Quelques r\'{e}sulatats de {V}i\v{s}ik sur les probl\`emes
                  elliptiques nonlin\'{e}aires par les m\'{e}thodes de
                  {M}inty-{B}rowder},
  Journal       = {Bull. Soc. Math. France},
  Volume        = {93},
  Year          = {1965},
  Pages         = {97--107},
  URL           = {http://www.numdam.org/item?id=BSMF_1965__93__97_0}
}

@TechReport{      Reed.Hill:73,
  Author        = {Reed, W. H. and Hill, T. R.},
  Title         = {Triangular mesh methods for the neutron transport
                  equation},
  Institution   = {Los Alamos Scientific Laboratory},
  Year          = {1973},
  Number        = {LA-UR-73-0479},
  URL           = {http://lib-www.lanl.gov/cgi-bin/getfile{?}00354107.pdf},
  Address       = {Los Alamos, NM}
}

\end{document}